\documentclass[10pt]{article}
\usepackage{amsmath,amsthm,amssymb,latexsym,enumerate}
\usepackage{color}
\newtheorem{thm}{Theorem}[section]
\newtheorem{cor}[thm]{Corollary}
\newtheorem{lem}[thm]{Lemma}
\newtheorem{prop}[thm]{Proposition}

\theoremstyle{definition}

\allowdisplaybreaks

\numberwithin{equation}{section}

\begin{document}

\title{\bf On an effect of inhomogeneous constraints\\for a maximizing problem of the Sobolev embedding associated with the space\\of bounded variation}
\author{Michinori Ishiwata\,$^1$ and Hidemitsu Wadade\,$^2$}
\date{\small\it $^1$Department of Systems Innovation
Graduate School of Engineering Science, Osaka University,\\
1-3 Machikaneyama, Toyonaka, Osaka 560-8531, Japan\vspace{.3cm}\\
$^2$Faculty of Mechanical Engineering, Institute of Science and Engineering, Kanazawa University,\\
Kakuma, Kanazawa, Ishikawa 920-1192, Japan}

\maketitle

\begin{abstract}
In this paper, we consider a maximizing problem associated with the Sobolev type embedding 
$BV(\Bbb R^N)\hookrightarrow L^r(\Bbb R^N)$ for $1\leq r\leq1^*:=\frac{N}{N-1}$ 
with $N\geq 2$ as follows : for given $\alpha>0$, 
$$
D_\alpha(a,b,q):=\sup_{u\in BV(\Bbb R^N),\atop \|u\|_{TV}^a+\|u\|_{1}^b=1}
\left(\|u\|_{1}+\alpha\|u\|_{q}^q\right), 
$$
where $1<q\leq 1^*$, $a, b>0$, we show that, although the maximizing problem 
associated with $D_\alpha(a,b,1^*)$ suffers from both of the non-compactness of $BV\hookrightarrow L^1$ and $BV\hookrightarrow L^{1^*}$ called vanishing and concentrating phenomena, there exists a maximizer for some range of $a$, $b$. 
Furthermore, we show that 
any maximizer $u\in BV$ of $D_\alpha(a,b,q)$ must be given by a characteristic function on a ball.  

\medskip


\medskip

\noindent 
{\it Key words\,:\,Sobolev's embedding\,;\,maximizing problem\,;\,space of bounded variation\,;\\Gagliardo-Nirenberg inequality\,;\,isoperimetrical inequality} 
\end{abstract}

\section{Introduction and main results}
\noindent

In this paper, we consider a maximizing problem associated with 
Sobolev type embedding $BV(\Bbb R^N)\hookrightarrow L^r(\Bbb R^N)$ for $1\leq r\leq 1^*:=\frac{N}{N-1}$ with $N\geq 2$, where $BV$ denotes the space of bounded variation, 
see \cite{G} and Section 2. The inequality associated with the embedding $BV\hookrightarrow L^{1^*}$ is Mazya's inequality with its best-constant $E$ given by 
\begin{align}\label{mazya}
E:=\sup_{u\in BV\setminus\{0\}}\left(\frac{\|u\|_{1^*}}{\|u\|_{TV}}\right)^{1^*}
=\left(\frac{1}{N^{N-1}\omega_{N-1}}\right)^{\frac{1}{N-1}}, 
\end{align}
where $\omega_{N-1}$ denotes the surface area of the $N$-dimensional unit ball, see \cite{M}. 
It is well-known that \eqref{mazya} is 
equivalent to the isoperimetric inequality, 
and maximizers of $E$ consist of functions of the form $\lambda\chi_{B}\in BV$ with $\lambda\in\Bbb R\setminus\{0\}$ and a ball $B\subset\Bbb R^N$. A variational problem investigated in this paper is formulated as follows : for given $\alpha>0$, 
\begin{align}\label{basic-pro}
D_\alpha:=\sup_{u\in BV, \,\,\|u\|_{TV}^a+\|u\|_{1}^b=1}\left(
\|u\|_{1}+\alpha \|u\|_{q}^q
\right), 
\end{align}
where $1<q\leq 1^*$ and $a, b>0$. Especially for the critical case $q=1^*$, the maximizing problem associated with $D_\alpha$ suffers from both of the non-compactness of $BV\hookrightarrow L^1$ and $BV\hookrightarrow L^{1^*}$ called vanishing and concentrating phenomena, respectively. 
One of our goals is to clarify an effect of the exponents $a$ and $b$ in 
the inhomogeneous constraints on the (non-)attainability of $D_\alpha$. 

\medskip

The attainability of maximizing problems corresponding to 
the Sobolev embedding $W^{1,p}\hookrightarrow L^r$, where $1<p<N$, 
$p\leq r\leq p^*:=\frac{Np}{N-p}$, were studied in \cite{IW2, N2}. 
The authors in \cite{IW2} treated the variational problem given by 
\begin{align*}
\sup_{u\in W^{1,p}, \,\,\|\nabla u\|_p^a+\|u\|_p^b=1}\left(
\|u\|_p^p+\alpha \|u\|_{q}^q
\right), 
\end{align*}
where $p<q<p^*$ and $a,b>0$. This problem contains a difficulty coming from the 
non-compactness of $W^{1,p}\hookrightarrow L^p$ 
due to a vanishing phenomenon. 
After that, the author in \cite{N2} considered the same problem for the critical case $q=p^*$. 
In this case, the problem becomes more complicated since one needs to exclude 
both of vanishing and concentrating behaviors of a maximizing sequence 
due to the non-compact embeddings $W^{1,p}\hookrightarrow L^p$ and $W^{1,p}\hookrightarrow L^{p^*}$, respectively. 
The usual way in attacking this problem will be to compute the thresholds with respect to  vanishing and concentrating phenomena and to investigate behaviors of a maximizing sequence in order to recover the compactness of the functional, 
which was a strategy used in \cite{IW2}. However, the author in \cite{N2} gave 
an alternative way in discussing the problem without a use 
of the variational method directly. A main key used in \cite{N2} is to give 
another expression of the functional in terms of the corresponding $1$-dimensional function 
by a scaling argument.  Based on these known results, 
we consider the remaining case $p=1$, which leads to the problem \eqref{basic-pro}. 
In fact, we observe that the method used in \cite{N2} can work for the marginal case $p=1$ 
by replacing $W^{1,1}$ with $BV$. Also, as an advantage of the case $p=1$, 
we know the exact forms of maximizers of $E$ through the isoperimetric inequality, 
and as a result, we obtain a characterization of maximizers of $D_\alpha$, 
see Theorem \ref{thm5}. 

\medskip

In order to state our main results, we start from the problem \eqref{basic-pro} with the subcritical case $1<q<1^*$. 
In this case, the embedding $BV_{rad}\hookrightarrow L^q$ is compact, 
where $BV_{rad}$ denotes the set of radially symmetric functions in $BV$, 
and hence, the term $\|u\|_{q}$ in the functional will make an aid 
to admit a maximizer of $D_\alpha$, see \cite{FP}. 
On the other hand, $D_\alpha$ suffers from the non-compactness 
of $BV\hookrightarrow L^1$, which comes from 
the scaling $u_n(x):=\frac{1}{n^N}u(\frac{x}{n})$ with a fixed 
$u\in BV\setminus\{0\}$. In general, we call $\{u_n\}_n\subset BV$ ``\,a vanishing sequence\,'' if $\{u_n\}_n$ satisfies the conditions : 
$$
\sup_n\|u_n\|_{BV}<\infty,\quad\inf_n\|u_n\|_{1}>0,\quad\lim_{n\to\infty}\|u_n\|_{TV}=0. 
$$
We also introduce the value $\alpha_v=\alpha_v(a,b,q)\in[0,\infty)$ defined by 
\begin{align*}
\alpha_v:=\displaystyle\inf_{u\in BV, \,\,\|u\|_{TV}^a+\|u\|_{1}^b=1}\frac{1-\|u\|_{1}}{\|u\|_{q}^q}. 
\end{align*}
If there exists a maximizing sequence $\{u_n\}_n$ of $D_\alpha$ 
such that $\{u_n\}_n$ is also a vanishing sequence, 
we easily see $D_\alpha\leq1$. On the other hand, since $\alpha>\alpha_v$ is equivalent to $D_\alpha>1$, the value $\alpha_v$ is expected to be the threshold of $\alpha$ on the attainability of $D_\alpha$. 
Our first result is stated as follows : 

\begin{thm}\label{thm1}
Let $1<q<1^*$, $a>0$ and $b>0$.

\medskip

\noindent 
{\bf (Non-threshold case $\alpha\ne\alpha_v$)}

\medskip

\noindent
{\rm (i)} Let $a>N(q-1)$. Then there holds $\alpha_v=0$, and $D_\alpha$ is attained for $\alpha>0$.

\medskip

\noindent
{\rm(ii)} Let $a\leq N(q-1)$. Then there holds $\alpha_v>0$, and $D_\alpha$ is attained for $\alpha>\alpha_v$, while $D_\alpha$ is not attained for $\alpha<\alpha_v$. 

\medskip

\noindent
{\bf(Threshold case $\alpha=\alpha_v$)}

\medskip

\noindent
{\rm (iii)} Let $a<N(q-1)$, or let $a=N(q-1)$, $\frac{2N-1}{2(N-1)}<q<1^*$ and $b<b_0:=(q-1)(N-1)-\left(N-(N-1)q\right)$. 
Then $D_{\alpha_v}$ is attained.

\medskip

\noindent
{\rm (iv)} Let $a=N(q-1)$ and $\begin{cases}
&\hspace{-.2cm}\frac{2N-1}{2(N-1)}<q<1^*\text{ \,and \,}b\geq b_0,\\
&\hspace{-.2cm}\text{or \,}1<q\leq\frac{2N-1}{2(N-1)}. 
\end{cases}$ Then $D_{\alpha_v}$ is not attained.
\end{thm}

Next, we estimate the value $\alpha_v$. To this end, we introduce 
the best-constant of the Gagliardo-Nirenberg type inequality $E_q$ : 
$$
E_q:=\sup_{u\in BV\setminus\{0\}}\frac{
\|u\|_{q}^q
}{
\|u\|_{1}^{q-(q-1)N}\|u\|_{TV}^{(q-1)N}. 
}
$$
One can calculate $E_q=\left(\frac{1}{N^{N-1}\omega_{N-1}}\right)^{q-1}$ 
and remark that $E_{1^*}=E$ is Mazya's best-constant, see Proposition \ref{GN-attain} (i).  
By means of $E_q$, the value $\alpha_v$ is estimated as follows : 

\begin{thm}\label{thm2}

\noindent
Let $1<q<1^*$, $a>0$ and $b>0$. 

\medskip

\noindent
{\rm(i)} There hold $
\alpha_v
\begin{cases}
=0\quad\text{when \,}a>N(q-1),\\
>0\quad\text{when \,}a\leq N(q-1). 
\end{cases}
$

\medskip

\noindent
{\rm(ii)} Let $a=N(q-1)$. Then there hold 
\begin{align*}
\begin{cases}
&\hspace{-.3cm}
\alpha_v=\frac{1}{b E_q}\text{ \,when }
\begin{cases}
&\hspace{-.2cm}\frac{2N-1}{2(N-1)}<q<1^*\text{ \,and \,}b\geq b_0,\\
&\hspace{-.2cm}\text{or \,}1<q\leq\frac{2N-1}{2(N-1)},
\end{cases}\vspace{.2cm}\\
&\hspace{-.3cm}0<\alpha_v<\frac{1}{b E_q}
\text{ \,when \,}
\frac{2N-1}{2(N-1)}<q<1^*\text{ \,and \,}b<b_0. 
\end{cases}
\end{align*}

\noindent
{\rm(iii)} {\bf(Asymptotic behaviors of $\alpha_v$ on the parameters $a$ and $b$)}

\medskip

\noindent
{\rm (a)} There holds $\lim_{a\downarrow 0}\alpha_v=\infty$. 

\medskip

\noindent
{\rm (b)} Let $\frac{2N-1}{2(N-1)}<q<1^*$ and $b\geq b_0$, or let $1<q\leq\frac{2N-1}{2(N-1)}$. 
Then there holds $\lim_{a\uparrow N(q-1)}\alpha_v=\frac{1}{b E_q}$. 

\medskip

\noindent
{\rm(c)} Let $a\leq N(q-1)$. Then there hold $\lim_{b\downarrow 0}\alpha_v=\infty$ 
and $\lim_{b\to\infty}\alpha_v=0$. 

\medskip

\noindent
{\rm(d)} Let $a=N(q-1)$ and $\frac{2N-1}{2(N-1)}<q<1^*$. Then there holds
$\lim_{b\uparrow b_0}\alpha_v=\frac{1}{b_0 E_q}$. 
\end{thm}

We proceed to the critical case $q=1^*$. 
In this case, $D_\alpha(a,b,1^*)$ suffers from 
the non-compactness of not only $BV\hookrightarrow L^1$ but also $BV\hookrightarrow L^{1^*}$. The latter non-compactness comes from the scaling 
$u_n(x):=n^{N-1}u(nx)$ with a fixed $u\in BV\setminus\{0\}$. 
In general, we call $\{u_n\}_n\subset BV$ ``\,a concentrating sequence\,'' if 
$\{u_n\}_n$ satisfies the conditions : 
$$
\sup_n\|u_n\|_{BV}<\infty,\quad \inf_n\|u_n\|_{1^*}>0,\quad\lim_{n\to\infty}\|u_n\|_{1}=0. 
$$
We also introduce the value $\alpha_c=\alpha_c(a,b)\in(0,\infty]$ defined by
$$
\alpha_c:=\displaystyle\sup_{u\in BV, \,\,\|u\|_{TV}^a+\|u\|_{1}^b=1}\frac{
\|u\|_{1}
}{E-\|u\|_{1^*}^{1^*}}, 
$$
where note $E-\|u\|_{1^*}^{1^*}\geq E\left(1-\|u\|_{TV}^{1^*}\right)>0$ since $0<\|u\|_{TV}<1$. 
If there exists a maximizing sequence $\{u_n\}_n$ of $D_\alpha$ 
such that $\{u_n\}_n$ is also a concentrating sequence, 
it is easy to see $D_\alpha\leq\alpha E$. On the other hand, since $\alpha<\alpha_c$ is equivalent to $D_\alpha>\alpha E$, the value $\alpha_c$ is expected to be the threshold of $\alpha$ on the attainability of $D_\alpha$ regarding to the concentrating phenomenon. 
In fact, we can show that $D_\alpha$ with $\alpha$ in the region $(\alpha_v,\alpha_c)$ admits a maximizer whenever $\alpha_v<\alpha_c$, see Lemma \ref{vani-conce-ge} (iii). 
We now state the attainability result on $D_\alpha=D_\alpha(a,b,1^*)$ :

\begin{thm}\label{thm3}
Let  $a>0$ and $b>0$.

\medskip

\noindent 
{\bf(Non-threshold case $\alpha\ne\alpha_v$ and $\alpha\ne\alpha_c$)}

\medskip

\noindent
{\rm(i)} Let $a>1^*$ and $b>1$. Then there hold $\alpha_v=0$ and $\alpha_c=\infty$, 
and $D_\alpha$ is attained for $\alpha>0$.

\medskip

\noindent
{\rm(ii)} Let $a>1^*$ and $b\leq 1$. Then there hold $\alpha_v=0$ and $\alpha_c<\infty$, 
and $D_\alpha$ is attained for $0<\alpha<\alpha_c$, 
while $D_\alpha$ is not attained for $\alpha>\alpha_c$. 

\medskip

\noindent
{\rm(iii)} Let $a\leq 1^*$ and $b>1$. Then there hold $\alpha_v>0$ and $\alpha_c=\infty$, 
and $D_\alpha$ is attained for $\alpha>\alpha_v$, while $D_\alpha$ is not attained for $\alpha<\alpha_v$. 

\medskip

\noindent
{\rm(iv)} Let $a\leq 1^*$ and $b\leq 1$. Then there holds $0<\alpha_v=\alpha_c<\infty$, 
and $D_\alpha$ is not attained for $\alpha\ne\alpha_v(=\alpha_c)$. 

\medskip

\noindent
{\bf(Threshold case $\alpha=\alpha_v$ or $\alpha=\alpha_c$)}

\medskip

\noindent
{\rm(v)} Let $a>1^*$. Then $D_{\alpha_c}\begin{cases}
&\text{ is attained when $b<1$},\\
&\text{ is not attained when $b=1$}. 
\end{cases}
$

\medskip

\noindent
{\rm(vi)} Let $a=1^*$. Then $D_{\alpha_v}\begin{cases}
&\text{ is attained when $b=1$},\\
&\text{ is not attained when $b\ne1$}. 
\end{cases}$

\medskip

\noindent
{\rm(vii)} Let $a<1^*$. Then $D_{\alpha_v}\begin{cases}
&\text{ is attained when $b>1$},\\
&\text{ is not attained when $b\leq1$}. 
\end{cases}$
\end{thm}

Next, we estimate $\alpha_v$ and $\alpha_c$ by means of $E$ as follows : 

\begin{thm}\label{thm4}

\noindent
Let $a>0$ and $b>0$. 

\medskip

\noindent
{\rm(i)} Let $a>1^*$. Then there hold $\alpha_v=0$ 
and $\begin{cases}
&\alpha_c=\infty\text{ \,when \,}b>1,\\
&\frac{1}{E}<\alpha_c<\infty\text{ \,when \,}b\leq 1.
\end{cases}
$

\noindent
In particular, there holds $\alpha_c=\frac{a}{1^*E}$ when $b=1$.

\medskip

\noindent
{\rm(ii)} Let $a\leq 1^*$. Then there hold $
\begin{cases}
&0<\alpha_v<\frac{1}{E}\text{ \,and \,}\alpha_c=\infty\text{ \,when \,}b>1,\\
&\alpha_v=\alpha_c=\frac{1}{E}\text{ \,when \,}b\leq1. 
\end{cases}
$

\noindent
In particular, there holds $\alpha_v=\frac{1}{bE}$ when $a=1^*$ and $b>1$. 

\medskip

\noindent
{\rm(iii)} {\bf(Asymptotic behaviors of $\alpha_v$ and $\alpha_c$ on the parameters $a$ and $b$)}

\medskip

\noindent
{\rm(a)} Let $b>1$. Then there hold $\lim_{a\downarrow 0}\alpha_v=\frac{1}{E}$ and $\lim_{a\uparrow 1^*}\alpha_v=\frac{1}{bE}$. 

\medskip

\noindent
{\rm(b)} Let $b\leq1$. Then there hold $\lim_{a\downarrow 1^*}\alpha_c=\frac{1}{E}$ and $\lim_{a\to\infty}\alpha_c=\infty$. 

\medskip

\noindent
{\rm(c)} Let $a>1^*$. Then there hold $\lim_{b\downarrow 0}\alpha_c=\frac{1}{E}$ 
and $\lim_{b\uparrow 1}\alpha_c=\frac{a}{1^*E}$. 

\medskip

\noindent
{\rm(d)} Let $a\leq 1^*$. Then there hold $\lim_{b\downarrow 1}\alpha_v=\frac{1}{E}$
and $\lim_{b\to\infty}\alpha_v=0$. 
\end{thm}

In the end, we characterize the set of all maximizers of $D_\alpha$ for $1<q\leq 1^*$ 
by means of the corresponding $1$-dimensional function : 

\begin{thm}\label{thm5}
Let $1<q\leq 1^*$, $\alpha>0$, $a>0$ and $b>0$. 
Introduce $1$-dimensional functions $r(t)$, $\mu(t)$ and $f_\alpha(t)$ by
\begin{align*}
\begin{cases}
&r(t):=\frac{N}{t^{\frac{1}{a}}(t+1)^{\frac{a-b}{ab}}},\vspace{.2cm}\\
&\mu(t):=\frac{t^{\frac{N}{a}}(t+1)^{\frac{(a-b)N}{ab}-\frac{1}{b}}}{\omega_{N-1}N^{N-1}},\vspace{.2cm}\\
&f_\alpha(t):=\displaystyle\frac{
(1+t)^{(q-1)(\frac{N}{a}-\frac{N-1}{b})}+\alpha E_q\,t^{\frac{(q-1)N}{a}}
}{
(1+t)^{\frac{(q-1)N}{a}+\frac{N-q(N-1)}{b}}
}
\end{cases}
\end{align*}
for $t>0$. Then it holds $D_\alpha=\sup_{t>0}f_\alpha(t)$. 
Furthermore, letting $\Sigma$ and $\Pi$ be sets defined by 
\begin{align*}
\begin{cases}
&\Sigma:=\{
u_0\in BV\,|\,u_0\text{ is a maximizer of }D_\alpha
\},\\
&\Pi:=\{
t_0>0\,|\,\text{$t_0$ is a maximal point of $\sup_{t>0}f_\alpha(t)$}
\},
\end{cases}
\end{align*}
we obtain 
$
\Sigma=\{
\pm\mu(t_0)\chi_{B_{r(t_0)}(x_0)}\in BV\,|\, t_0\in\Pi\text{ and }x_0\in\Bbb R^N
\}$. 
\end{thm}

\noindent
Theorem \ref{thm5} is essentially proved by a scaling argument in \cite{N2} 
together with the fact that maximizers of $E_q$ consist of functions of the form 
$u=\lambda\chi_B$ with $\lambda\in\Bbb R\setminus\{0\}$ and a ball $B\subset\Bbb R^N$, see Proposition \ref{GN-attain} (ii), namely the information on maximizers of $E_{1^*}$ (the isoperimetric inequality) is transmitted to $E_q$ for any $1<q\leq 1^*$. On the other hand, it seems to be difficult to obtain a similar characterization to the problem based on $W^{1,p}$ with $1<p<N$  
since we do not know the relation between maximizers of the Sobolev inequality (called Talenti's function) and those of the corresponding Gagliardo-Nirenberg inequality. 
 
\medskip

For the limiting case $p=N$, a maximizing problem on $W^{1,N}$ corresponding to \eqref{basic-pro}  was considered in \cite{IW}. 
As another characterization of Sobolev's embedding in this case, 
we know the Moser-Trudinger type inequalities. 
Attainability problems associated with those inequalities 
also have been investigated in rich literature. 
Among others, we refer to \cite{IIW, I, LR, OST, R} and related works \cite{L, L2, LLZ, Li, LY, N}, in which similar problems to $D_\alpha$ were studied. 

\medskip

This paper is organized as follows. Section 2 is devoted to prepare preliminary facts 
and to prove Theorem \ref{thm5}. We show Theorem \ref{thm1}-\ref{thm2} and Theorem \ref{thm3}-\ref{thm4} in Section 3 and Section 4, respectively. Throughout the paper, the notation $\|\cdot\|_{p}$ denotes the standard $L^p$-norm. We pass to subsequences freely.

\section{Preliminaries}
\noindent

In this section, we collect several lemmas needed for the proofs of main theorems. 
First, we recall the definition of the space of bounded variation $BV$. 
$BV$ is a Banach space endowed with the norm $\|u\|_{BV}:=\|u\|_{TV}+\|u\|_{1}$, 
where the total variation $\|u\|_{TV}$ is given by 
$$
\|u\|_{TV}:=\sup\left\{
\int_{\Bbb R^N}u\operatorname{div}\psi\,\bigg|\,\psi=\{\psi_1,\cdots,\psi_N\}\subset C^1_c, \,
\|\psi\|_{\infty}:=\left\|\left(\sum_{i=1}^N|\psi_i|^2\right)^{\frac{1}{2}}\right\|_{\infty}\leq 1
\right\}. 
$$
The Sobolev type embedding on $BV$ states $W^{1,1}\hookrightarrow BV\hookrightarrow L^r$ for $1\leq r\leq 1^*$. 
We introduce $E_q$ and $\tilde E_q$ by 
\begin{align*}
\begin{cases}
&E_q=E_q(u):=\sup_{u\in BV\setminus\{0\}}\frac{
\|u\|_{q}^q
}{
\|u\|_{1}^{q-(q-1)N}\|u\|_{TV}^{(q-1)N}
},\\
&\tilde E_q=\tilde E_q(u):=\sup_{u\in W^{1,1}\setminus\{0\}}\frac{
\|u\|_{q}^q
}{
\|u\|_{1}^{q-(q-1)N}\|\nabla u\|_{1}^{(q-1)N}
}
\end{cases}
\end{align*}
and similarly $D_\alpha$ and $\tilde D_\alpha$ by 
\begin{align*}
\begin{cases}
&D_\alpha:=\sup_{\|u\|_{TV}^a+\|u\|_{1}^b=1}\left(\|u\|_{1}+\alpha\|u\|_{q}^q\right),\\
&\tilde D_\alpha:=\sup_{\|\nabla u\|_{1}^a+\|u\|_{1}^b=1}\left(
\|u\|_{1}+\alpha\|u\|_{q}^q
\right)
\end{cases}
\end{align*}
for $1<q\leq 1^*$ and $a,b,\alpha>0$. 

\begin{prop}\label{GN-attain}
Let $1<q\leq 1^*$. 

\medskip

\noindent
{\rm(i)} There holds $E_q=\tilde E_q=\left(\frac{1}{N^{N-1}\omega_{N-1}}\right)^{q-1}$. 

\medskip

\noindent
{\rm(ii)} $E_q$ is attained by functions of the form $u=\lambda\chi_B\in BV$ with $\lambda\in \Bbb R\setminus\{0\}$ and a ball $B\subset \Bbb R^N$. 
Moreover, the maximizer of $E_q$ necessarily has this form. 

\medskip

\noindent
{\rm(iii)} $\tilde E_q$ is not attained in $W^{1,1}\setminus\{0\}$. 
\end{prop}

\begin{proof}
First, recall the facts that it holds  
$E_{1^*}=\frac{1}{N\omega_{N-1}^{\frac{1}{N-1}}}$ and $E_{1^*}$ is attained only by functions of the form $u=\lambda\chi_B\in BV$ with $\lambda\in \Bbb R\setminus\{0\}$ and a ball $B\subset \Bbb R^N$. 

\medskip

\noindent
(i) By H\"older's inequality and Mazya's inequality, we have for $u\in BV$
\begin{align*}
&\|u\|_{q}^q\leq\|u\|_{1}^{q-(q-1)N}\|u\|_{1^*}^{(q-1)N}\\
&\leq\|u\|_{1}^{q-(q-1)N}\left(\frac{1}{N^{\frac{N-1}{N}}\omega_{N-1}^{\frac{1}{N}}}\|u\|_{TV}\right)^{(q-1)N}
=\left(\frac{1}{N^{N-1}\omega_{N-1}}\right)^{q-1}\|u\|_{1}^{q-(q-1)N}\|u\|_{TV}^{(q-1)N}, 
\end{align*}
which implies $E_q\leq\left(\frac{1}{N^{N-1}\omega_{N-1}}\right)^{q-1}$. 
Let $u_0=\chi_{B_1(0)}\in BV$. Then we can compute 
$\|u_0\|_{1}=\|u_0\|_{q}^q=\frac{\omega_{N-1}}{N}$ and $\|u_0\|_{TV}=\omega_{N-1}$, 
and then we observe $E_q(u_0)=\left(\frac{1}{N^{N-1}\omega_{N-1}}\right)^{q-1}$. 
Hence, $u_0$ is a maximizer of $E_q$ and it follows $E_q=\left(\frac{1}{N^{N-1}\omega_{N-1}}\right)^{q-1}$. 

\medskip

Next, we prove $E_q=\tilde E_q$. 
It is enough to show $E_q\leq \tilde E_q$ since the converse inequality is obtained 
by the facts $W^{1,1}\subset BV$ and $\|\nabla u\|_{1}=\|u\|_{TV}$ for $u\in W^{1,1}$. 
Let $u_0\in BV\setminus\{0\}$ be a maximizer of $E_q$, 
where note that the existence of $u_0$ has been already established as above. 
By an approximation argument, there exists a sequence $\{u_n\}_{n=1}^\infty\subset BV\cap C^\infty$ 
such that $u_n\to u_0$ in $L^1$ and $\|u_n\|_{TV}\to \|u_0\|_{TV}$, 
and up to a subsequence, $u_n\to u_0$ a.e. on $\Bbb R^N$. 
We observe that $u_n\in W^{1,1}$ with $\|u_n\|_{TV}=\|\nabla u_n\|_{1}$. 
Indeed, by using the fact that there holds $\|v\|_{TV(\Omega)}=\int_\Omega|\nabla v|$ for any $v\in BV(\Omega)\cap C^\infty(\Omega)$ with a bounded domain having its sufficiently smooth boundary, 
we see 
\begin{align*}
&\|u_n\|_{TV}=\sup_{R>0}\|u_n\|_{TV(B_R(0))}
=\sup_{R>0}\int_{B_R(0)}|\nabla u_n|=\lim_{R\to\infty}\int_{B_R(0)}|\nabla u_n|=\|\nabla u_n\|_{1}<\infty, 
\end{align*}
where the last equality is shown by Lebesgue's monotone convergence theorem. 
Then it holds $u_n\ne 0$ in $W^{1,1}$ for large $n$ since 
$\|\nabla u_n\|_{1}=\|u_n\|_{TV}\to \|u_0\|_{TV}>0$ as $n\to\infty$. 
Now we see by the convergences of $u_n$ together with Fatou's lemma, 
\begin{align*}
E_q=E_q(u_0)\leq\liminf_{n\to\infty}E_q(u_n)\leq\limsup_{n\to\infty}E_q(u_n)=\limsup_{n\to\infty}\tilde E_q(u_n)
\leq \tilde E_q. 
\end{align*}
Thus the assertion (i) has been proved.

\medskip

\noindent
(ii) Let $u_0=\lambda\chi_B\in BV$ for $\lambda\in\Bbb R\setminus\{0\}$ and a ball $B=B_R(x_0)$ 
with a radius $R>0$ centered at $x_0\in\Bbb R^N$. 
Then we can compute 
\begin{align*}
\|u_0\|_{1}=|\lambda|R^N\frac{\omega_{N-1}}{N},\quad \|u_0\|_{q}^q=|\lambda|^q R^N\frac{\omega_{N-1}}{N}\quad\text{and}\quad
\|u_0\|_{TV}=|\lambda|R^{N-1}\omega_{N-1}, 
\end{align*}
and thus these relations together with the assertion (i) show 
$E_q(u_0)=\left(\frac{1}{N^{N-1}\omega_{N-1}}\right)^{q-1}=E_q$. 
Hence, $u_0$ is a maximizer of $E_q$. 

\medskip

Next, assume that $E_q$ is attained by $u_0\in BV\setminus\{0\}$. 
Then by H\"older's inequality and the assertion (i), we have 
\begin{align*}
&\left(\frac{1}{N^{N-1}\omega_{N-1}}\right)^{q-1}=E_q=E_q(u_0)\leq E_{1^*}(u_0) ^{(q-1)(N-1)}\leq E_{1^*}^{(q-1)(N-1)}
=\left(\frac{1}{N^{N-1}\omega_{N-1}}\right)^{q-1}, 
\end{align*}
which shows that $u_0$ is a maximizer of $E_{1^*}$. 
Hence, $u_0=\lambda \chi_{B}$ for some $\lambda\in\Bbb R\setminus\{0\}$ and a ball $B\subset\Bbb R^N$. 
The assertion (ii) has been proved. 

\medskip

\noindent
(iii) By contradiction, assume that $\tilde E_q$ is attained by $u_0\in W^{1,1}\setminus\{0\}$. 
Then the assertion (i) and the facts $W^{1,1}\subset BV$ and $\|\nabla u\|_{1}=\|u\|_{TV}$ 
for $u\in W^{1,1}$ imply that $u_0\in BV\setminus\{0\}$ is a maximizer of $E_q$. 
Then the assertion (ii) shows that $u_0=\lambda \chi_B$ for $\lambda\in\Bbb R\setminus\{0\}$ and a ball $B\subset\Bbb R^N$, 
which is a contradiction to $u_0\in W^{1,1}$. 
The assertion (iii) has been proved. 
\end{proof}

\begin{lem}\label{D-f-g}
Let $1<q\leq 1^*$, $\alpha>0$, $a>0$ and $b>0$. 
Then there hold $D_\alpha=\sup_{t>0}f_\alpha(t)$ and $\alpha_v=\frac{1}{E_q}\inf_{t>0}g(t)$  where for $t>0$, 
\begin{align*}
\begin{cases}
&f_\alpha(t):=\displaystyle\frac{
(1+t)^{(q-1)(\frac{N}{a}-\frac{N-1}{b})}+\alpha E_q\,t^{\frac{(q-1)N}{a}}
}{
(1+t)^{\frac{(q-1)N}{a}+\frac{N-q(N-1)}{b}}
},\\
&g(t):=\displaystyle\frac{
\left((1+t)^{\frac{1}{b}}-1\right)(1+t)^{(q-1)(\frac{N}{a}-\frac{N-1}{b})}
}{
t^{\frac{(q-1)N}{a}}
}. 
\end{cases}
\end{align*}
\end{lem}

\begin{proof}
For $u\in BV$ with $\|u\|_{TV}^a+\|u\|_{1}^b=1$, we see 
\begin{align*}
&\|u\|_{1}+\alpha\|u\|_{q}^q\leq\|u\|_{1}+\alpha E_q\|u\|_{1}^{q-(q-1)N}\|u\|_{TV}^{(q-1)N}\\
&=\frac{
\|u\|_{1}(\|u\|_{TV}^a+\|u\|_{1}^b)^{(q-1)(\frac{N}{a}-\frac{N-1}{b})}
+\alpha E_q\|u\|_{1}^{q-(q-1)N}\|u\|_{TV}^{(q-1)N}
}{
(\|u\|_{TV}^a+\|u\|_{1}^b)^{\frac{(q-1)N}{a}+\frac{N-q(N-1)}{b}}
}\\
&=\frac{
\left(1+\frac{\|u\|_{TV}^a}{\|u\|_{1}^b}\right)^{(q-1)(\frac{N}{a}-\frac{N-1}{b})}
+\alpha E_q\left(\frac{\|u\|_{TV}^a}{\|u\|_{1}^b}\right)^{\frac{(q-1)N}{a}}
}{\left(1+\frac{\|u\|_{TV}^a}{\|u\|_{1}^b}\right)^{\frac{(q-1)N}{a}+\frac{N-q(N-1)}{b}}}\\
&=f_\alpha\left(\frac{\|u\|_{TV}^a}{\|u\|_{1}^b}\right)
\leq\sup_{t>0}f_\alpha(t),  
\end{align*}
which yields $D_\alpha\leq\sup_{t>0}f_\alpha(t)$. 
On the other hand, let $v\in BV\setminus\{0\}$ be a maximizer of $E_q$. 
The existence of $v$ was obtained by Proposition \ref{GN-attain} (ii). 
For $\lambda>0$, we define $v_\lambda(x):=K\lambda v(\lambda^{\frac{1}{N}}x)$, 
where $K=K(\lambda)>0$ is determined uniquely by 
\begin{align}\label{K-def1}
\|v_\lambda\|_{TV}^a+\|v_\lambda\|_{1}^b
=K^a\lambda^{\frac{a}{N}}\|v\|_{TV}^a+K^b\|v\|_{1}^b=1. 
\end{align}
Then we observe for $\lambda>0$, 
\begin{align*}
&D_\alpha\geq\|v_\lambda\|_{1}+\alpha\|v_\lambda\|_{q}^q\\
&=\frac{
\|v_\lambda\|_{1}(\|v_\lambda\|_{TV}^a+\|v_\lambda\|_{1}^b)^{(q-1)(\frac{N}{a}-\frac{N-1}{b})}
+\alpha E_q\|v_\lambda\|_{1}^{q-(q-1)N}\|v_\lambda\|_{TV}^{(q-1)N}
}{
(\|v_\lambda\|_{TV}^a+\|v_\lambda\|_{1}^b)^{\frac{(q-1)N}{a}+\frac{N-q(N-1)}{b}}
}\\
&=\frac{
\left(1+\frac{\|v_\lambda\|_{TV}^a}{\|v_\lambda\|_{1}^b}\right)^{(q-1)(\frac{N}{a}-\frac{N-1}{b})}
+\alpha E_q\left(\frac{\|v_\lambda\|_{TV}^a}{\|v_\lambda\|_{1}^b}\right)^{\frac{(q-1)N}{a}}
}{
\left(1+\frac{\|v_\lambda\|_{TV}^a}{\|v_\lambda\|_{1}^b}\right)^{\frac{(q-1)N}{a}+\frac{N-q(N-1)}{b}}
}\\
&=f_\alpha\left(\frac{\|v_\lambda\|_{TV}^a}{\|v_\lambda\|_{1}^b}\right)
=f_\alpha\left(K^{a-b}\lambda^{\frac{a}{N}}\frac{\|v\|_{TV}^a}{\|v\|_{1}^b}\right)
=f_\alpha\left(\frac{1}{K^b\|v\|_{1}^b}-1\right). 
\end{align*}
By the equation \eqref{K-def1}, we see that $K=K(\lambda)$ is a continuous 
function on $(0,\infty)$ satisfying $K<\frac{1}{\|v\|_{1}}$ for $\lambda>0$, $\lim_{\lambda\downarrow 0}K=\frac{1}{\|v\|_{1}}$ and $\lim_{\lambda\to\infty}K=0$. Thus we obtain 
$D_\alpha\geq\sup_{\lambda>0}f_\alpha\left(\frac{1}{K^b\|v\|_{1}^b}-1\right)
=\sup_{t>0}f_\alpha(t)$. Thus we have proved $D_\alpha=\sup_{t>0}f_\alpha(t)$. 

\medskip

Next, for $u\in BV$ with $\|u\|_{TV}^a+\|u\|_{1}^b=1$, we see 
\begin{align*}
&\frac{1-\|u\|_{1}}{\|u\|_{q}^q}\geq\frac{1-\|u\|_{1}}{E_q\|u\|_{1}^{q-(q-1)N}\|u\|_{TV}^{(q-1)N}}\\
&=\frac{
\left(\left(1+\frac{\|u\|_{TV}^a}{\|u\|_{1}^b}\right)^{\frac{1}{b}}-1\right)
\left(
1+\frac{\|u\|_{TV}^a}{\|u\|_{1}^b}
\right)^{(q-1)(\frac{N}{a}-\frac{N-1}{b})}
}{
E_q \left(\frac{\|u\|_{TV}^a}{\|u\|_{1}^b}\right)^{\frac{(q-1)N}{a}}
}=\frac{1}{E_q}g\left(
\frac{\|u\|_{TV}^a}{\|u\|_{1}^b}
\right)\geq\frac{1}{E_q}\inf_{t>0}g(t), 
\end{align*}
and thus there holds $\alpha_v\geq\frac{1}{E_q}\inf_{t>0}g(t)$. 
On the other hand, let $v\in BV\setminus\{0\}$ be a maximizer of $E_q$ 
and define $v_\lambda$ as above. Then we see for $\lambda>0$, 
\begin{align*}
&\alpha_v\leq\frac{1-\|v_\lambda\|_{1}}{\|v_\lambda\|_{q}^q}
=\frac{1-\|v_\lambda\|_{1}}{E_q\|v_\lambda\|_{1}^{q-(q-1)N}\|v_\lambda\|_{TV}^{(q-1)N}}\\
&=\frac{
\left(\left(1+\frac{\|v_\lambda\|_{TV}^a}{\|v_\lambda\|_{1}^b}\right)^{\frac{1}{b}}-1\right)
\left(1+\frac{\|v_\lambda\|_{TV}^a}{\|v_\lambda\|_{1}^b}\right)^{(q-1)(\frac{N}{a}-\frac{N-1}{b})}
}{
E_q\left(\frac{\|v_\lambda\|_{TV}^a}{\|v_\lambda\|_{1}^b}\right)^{\frac{(q-1)N}{a}}
}\\
&=\frac{1}{E_q}g\left(\frac{\|v_\lambda\|_{TV}^a}{\|v_\lambda\|_{1}^b}\right)
=\frac{1}{E_q}g\left(K^{a-b}\lambda^{\frac{a}{N}}\frac{\|v\|_{TV}^a}{\|v\|_{1}^b}\right)
=\frac{1}{E_q}g\left(\frac{1}{K^b\|v\|_{1}^b}-1\right), 
\end{align*}
and thus we get $\alpha_v\leq\frac{1}{E_q}\inf_{\lambda>0}g\left(\frac{1}{K^b\|v\|_{1}^b}-1\right)=\frac{1}{E_q}\inf_{t>0}g(t)$. 
Thus we have proved $\alpha_v=\frac{1}{E_q}\inf_{t>0}g(t)$. 
\end{proof}

\begin{lem}\label{D-imply-f(t)}
Let $1<q\leq 1^*$, $\alpha>0$, $a>0$ and $b>0$. Assume that $D_\alpha$ is attained by $u_0\in BV$. Then there exist $R>0$, $x_0\in\Bbb R^N$ and $\lambda_0\in\Bbb R\setminus\{0\}$ such that 
$u_0=\lambda_0\chi_{B_R(x_0)}$, where the coefficient $\lambda_0$ satisfies 
\begin{align}\label{lambda0-condition}
(|\lambda_0|R^{N-1}\omega_{N-1})^a+\left(|\lambda_0|R^N\frac{\omega_{N-1}}{N}\right)^b=1. 
\end{align}
In addition, $\sup_{t>0}f_\alpha(t)$ is attained at $t=(\frac{N}{R})^b(|\lambda_0|R^{N-1}\omega_{N-1})^{a-b}$. 
\end{lem}

\begin{proof}
By Lemma \ref{D-f-g}, we see 
\begin{align}
&\notag\sup_{t>0}f_\alpha(t)=D_\alpha=\|u_0\|_{1}+\alpha\|u_0\|_{q}^q
\leq\|u_0\|_{1}+\alpha E_q\|u_0\|_{1}^{q-(q-1)N}\|u_0\|_{TV}^{(q-1)N}\\
&\notag=\frac{
\|u_0\|_{1}\left(\|u_0\|_{TV}^a+\|u_0\|_{1}^b\right)^{(q-1)(\frac{N}{a}-\frac{N-1}{b})}
+\alpha E_q\|u_0\|_{1}^{q-(q-1)N}\|u_0\|_{TV}^{(q-1)N}
}{
\left(\|u_0\|_{TV}^a+\|u_0\|_{1}^b\right)^{\frac{(q-1)N}{a}+\frac{N-q(N-1)}{b}}
}\\
&\label{dfg-appli}=\frac{
\left(1+\frac{\|u_0\|_{TV}^a}{\|u_0\|_{1}^b}\right)^{(q-1)(\frac{N}{a}-\frac{N-1}{b})}+\alpha E_q\left(\frac{\|u_0\|_{TV}^a}{\|u_0\|_{1}^b}\right)^{\frac{(q-1)N}{a}}
}{\left(1+\frac{\|u_0\|_{TV}^a}{\|u_0\|_{1}^b}\right)^{\frac{(q-1)N}{a}+\frac{N-q(N-1)}{b}}}
=f_\alpha\left(\frac{\|u_0\|_{TV}^a}{\|u_0\|_{1}^b}\right)\leq\sup_{t>0}f_\alpha(t), 
\end{align}
which implies that $u_0$ is a maximizer of $E_q$. By applying Proposition \ref{GN-attain} (ii), 
we can write $u_0=\lambda_0\chi_{B_R(x_0)}$ for some $\lambda_0\in\Bbb R\setminus\{0\}$, $R>0$ and $x_0\in\Bbb R^N$. 
Moreover, since $\|u_0\|_{1}=|\lambda_0|R^N\frac{\omega_{N-1}}{N}$ and $\|u_0\|_{TV}=|\lambda_0|R^{N-1}\omega_{N-1}$, 
the coefficient $\lambda_0$ satisfies
\begin{align*}
\|u_0\|_{TV}^a+\|u_0\|_{1}^b=
(|\lambda_0|R^{N-1}\omega_{N-1})^a+\left(|\lambda_0|R^N\frac{\omega_{N-1}}{N}\right)^b=1. 
\end{align*}
In addition, the relation \eqref{dfg-appli} also implies that $\sup_{t>0}f_\alpha(t)$ is attained at 
\begin{align*}
t=\frac{\|u_0\|_{TV}^a}{\|u_0\|_{1}^b}=\left(\frac{N}{R}\right)^b(|\lambda_0|R^{N-1}\omega_{N-1})^{a-b}. 
\end{align*}
The proof of Lemma \ref{D-imply-f(t)} is complete. 
\end{proof}

\begin{cor}\label{non-attain-cor}
Let $1<q\leq 1^*$, $\alpha>0$, $a>0$ and $b>0$. Then $\tilde D_\alpha$ is not attained. 
\end{cor}

\begin{proof}
On the contrary, assume that $\tilde D_\alpha$ is attained by $u_0\in W^{1,1}$. 
Then recalling $\|\nabla u_0\|_{1}=\|u_0\|_{TV}$, we see that 
$u_0$ also becomes a maximizer of $D_\alpha$. Then by Lemma \ref{D-imply-f(t)}, 
we have $u_0=\lambda_0\chi_B$ with some $\lambda_0\in\Bbb R\setminus\{0\}$ 
and some ball $B\subset\Bbb R^N$, which is a contradiction to $u_0\in W^{1,1}$. 
\end{proof}

We are ready to prove Theorem \ref{thm5} :

\begin{proof}[{\rm \bf Proof of Theorem \ref{thm5}}]
First, we show $\Sigma\supset\{
\pm\mu(t_0)\chi_{B_{r(t_0)}(x_0)}\in BV\,|\, t_0\in\Pi\text{ and }x_0\in\Bbb R^N
\}$. To this end, let $t_0\in\Pi$, $x_0\in\Bbb R^N$ and $R>0$. 
Recall that $v:=\pm\chi_{B_R(x_0')}$ is a maximizer of $E_q$, 
where $x_0':=\frac{R t_0^{\frac{1}{a}}(t_0+1)^{\frac{a-b}{ab}}}{N}x_0$. 
For $\lambda>0$, define $v_\lambda(x):=K\lambda v(\lambda^{\frac{1}{N}}x)$, 
where $K=K(\lambda)>0$ is determined uniquely by 
\begin{align}\label{K-def}
\|v_\lambda\|_{TV}^a+\|v_\lambda\|_{1}^b=K^a\lambda^{\frac{a}{N}}\|v\|_{TV}^a+K^b\|v\|_{1}^b=1. 
\end{align}
Then we see 
\begin{align}\label{K-trans}
\frac{\|v_\lambda\|_{TV}^a}{\|v_\lambda\|_{1}^b}=K^{a-b}\lambda^{\frac{a}{N}}\frac{\|v\|_{TV}^a}{\|v\|_{1}^b}
=\frac{1}{K^b\|v\|_{1}^b}-1. 
\end{align}
Note that the relation \eqref{K-def} shows $\lim_{\lambda\downarrow 0}K=\frac{1}{\|v\|_{1}}$, $\lim_{\lambda\to\infty}K=0$ and 
\begin{align*}
K'=-\frac{
\frac{a}{N}K^a\lambda^{\frac{a}{N}-1}\|v\|_{TV}^a
}{
aK^{a-1}\lambda^{\frac{a}{N}}\|v\|_{TV}^a+bK^{b-1}\|v\|_{1}^b
}<0. 
\end{align*}
Hence, the relation \eqref{K-trans} implies that there exists $\lambda_0>0$ uniquely such that 
\begin{align}\label{lambda0t0}
\frac{\|v_{\lambda_0}\|_{TV}^a}{\|v_{\lambda_0}\|_{1}^b}=\frac{1}{K(\lambda_0)^b\|v\|_{1}^b}-1=t_0. 
\end{align}
Combining \eqref{K-def} with \eqref{lambda0t0}, we can compute 
\begin{align}\label{value-flam0}
&\lambda_0=\left(\frac{\|v\|_{1}}{\|v\|_{TV}}
t_0^{\frac{1}{a}}(t_0+1)^{\frac{a-b}{ab}}
\right)^N=\left(\frac{R}{N}t_0^{\frac{1}{a}}(t_0+1)^{\frac{a-b}{ab}}\right)^N,  
\end{align}
where we have used $\|v\|_{1}=R^N\frac{\omega_{N-1}}{N}$ and $\|v\|_{TV}=R^{N-1}\omega_{N-1}$. 
By Lemma \ref{D-f-g}, we see 
\begin{align*}
&\sup_{t>0}f_\alpha(t)=D_\alpha\geq\|v_{\lambda_0}\|_{1}+\alpha\|v_{\lambda_0}\|_{q}^q\\
&=\frac{
\|v_{\lambda_0}\|_{1}\left(\|v_{\lambda_0}\|_{TV}^a+\|v_{\lambda_0}\|_{1}^b\right)^{(q-1)(\frac{N}{a}-\frac{N-1}{b})}+\alpha E_q\|v_{\lambda_0}\|_{1}^{q-(q-1)N}\|v_{\lambda_0}\|_{TV}^{(q-1)N}
}{
\left(\|v_{\lambda_0}\|_{TV}^a+\|v_{\lambda_0}\|_{1}^b\right)^{\frac{(q-1)N}{a}+\frac{N-q(N-1)}{b}}
}\\
&=\frac{
\left(1+\frac{\|v_{\lambda_0}\|_{TV}^a}{\|v_{\lambda_0}\|_{1}^b}\right)^{(q-1)(\frac{N}{a}-\frac{N-1}{b})}+\alpha E_q\left(\frac{\|v_{\lambda_0}\|_{TV}^a}{\|v_{\lambda_0}\|_{1}^b}\right)^{\frac{(q-1)N}{a}}
}{
\left(
1+\frac{\|v_{\lambda_0}\|_{TV}^a}{\|v_{\lambda_0}\|_{1}^b}
\right)^{\frac{(q-1)N}{a}+\frac{N-q(N-1)}{b}}
}=f_\alpha\left(\frac{\|v_{\lambda_0}\|_{TV}^a}{\|v_{\lambda_0}\|_{1}^b}\right)\\
&=f_\alpha\left(\frac{1}{K(\lambda_0)^b\|v\|_{1}^b}-1\right)=f_\alpha(t_0)=\sup_{t>0}f_\alpha(t), 
\end{align*}
which implies that $v_{\lambda_0}$ is a maximizer of $D_\alpha$. Moreover, by \eqref{lambda0t0} and \eqref{value-flam0}, we can compute 
$v_{\lambda_0}=\pm\mu(t_0)\chi_{B_{r(t_0)}(x_0)}$. 

\medskip

Next, we show $\Sigma\subset\{
\pm\mu(t_0)\chi_{B_{r(t_0)}(x_0)}\in BV\,|\, t_0\in\Pi\text{ and }x_0\in\Bbb R^N
\}$. To this end, let $u_0\in BV$ be a maximizer of $D_\alpha$. 
Then by Lemma \ref{D-imply-f(t)}, we can write $u_0=\lambda_0\chi_{B_R(x_0)}$ with some $R>0$, $x_0\in\Bbb R^N$ and $\lambda_0\in\Bbb R\setminus\{0\}$, 
where $\lambda_0$ satisfies \eqref{lambda0-condition}. 
We take $t_0>0$ uniquely determined by the equation $R=r(t_0)$ 
and put $\nu:=\mu(t_0)$. 
Then we observe that $R$ and $\nu$ satisfy $(\nu R^{N-1}\omega_{N-1})^a+(\nu R^N\frac{\omega_{N-1}}{N})^b=1$, 
which implies $\nu=|\lambda_0|$ since $|\lambda_0|$ satisfies the same equation by \eqref{lambda0-condition}. 
Therefore, in order to complete the proof, it suffices to prove $t_0\in\Pi$. 
Noting that $u_0$ is a maximizer both of $D_\alpha$ and $E_q$ together with 
Lemma \ref{D-f-g}, we see
\begin{align*}
&\sup_{t>0}f_\alpha(t)=D_\alpha=\|u_0\|_{1}+\alpha\|u_0\|_{q}^q
=\|u_0\|_{1}+\alpha E_q\|u_0\|_{1}^{q-(q-1)N}\|u_0\|_{TV}^{(q-1)N}\\
&=\frac{
\|u_0\|_{1}\left(\|u_0\|_{TV}^a+\|u_0\|_{1}^b\right)^{(q-1)(\frac{N}{a}-\frac{N-1}{b})}
+\alpha E_q\|u_0\|_{1}^{q-(q-1)N}\|u_0\|_{TV}^{(q-1)N}
}{
\left(\|u_0\|_{TV}^a+\|u_0\|_{1}^b\right)^{\frac{(q-1)N}{a}+\frac{N-q(N-1)}{b}}
}\\
&=\frac{
\left(1+\frac{\|u_0\|_{TV}^a}{\|u_0\|_{1}^b}\right)^{(q-1)(\frac{N}{a}-\frac{N-1}{b})}
+\alpha E_q\left(\frac{\|u_0\|_{TV}^a}{\|u_0\|_{1}^b}\right)^{\frac{(q-1)N}{a}}
}{
\left(1+\frac{\|u_0\|_{TV}^a}{\|u_0\|_{1}^b}\right)^{\frac{(q-1)N}{a}+\frac{N-q(N-1)}{b}}
}=f_\alpha\left(\frac{\|u_0\|_{TV}^a}{\|u_0\|_{1}^b}\right)=f_\alpha(t_0), 
\end{align*}
where we have used $\frac{\|u_0\|_{TV}^a}{\|u_0\|_{1}^b}=t_0$. 
Hence, it follows $\sup_{t>0}f_\alpha(t)=f_\alpha(t_0)$, which means $t_0\in\Pi$. 
The proof of Theorem \ref{thm5} is complete. 
\end{proof}

\section{Proof of Theorems \ref{thm1}-\ref{thm2}}
\noindent

In this section, we shall prove Theorems \ref{thm1}-\ref{thm2}. 
We start from the following lemma : 

\begin{lem}\label{routine-attain-lem}
Let $1<q<1^*$, $a>0$ and $b>0$. 

\medskip

\noindent
{\rm(i)} Let $\alpha>\alpha_v$. Then $D_\alpha$ is attained. 

\medskip

\noindent
{\rm(ii)} Assume $\alpha_v>0$ and let $0<\alpha<\alpha_v$. Then $D_\alpha$ is not attained. 
\end{lem}

\begin{proof}
By Theorem \ref{thm5}, we see that $D_\alpha$ is attained if and only if $\sup_{t>0}f_\alpha(t)$ 
is attained. 

\medskip

\noindent
(i) Let $\alpha>\alpha_v$. Note that the condition $q<1^*$ shows $\lim_{t\to\infty}f_\alpha(t)=0$. By the assumption $\alpha>\alpha_v$ and Lemma \ref{D-f-g}, 
there exists $t_0>0$ such that $\alpha>\frac{1}{E_q}g(t_0)$, which implies $f_\alpha(t_0)>1=\lim_{t\downarrow 0}f_\alpha(t)$. Hence, $\sup_{t>0}f_\alpha(t)$ is attained. 

\medskip

\noindent
(ii) Assume $\alpha_v>0$ and let $0<\alpha<\alpha_v$. 
By contradiction, assume that there exists $t_0>0$ such that 
$\sup_{t>0}f_\alpha(t)=f_\alpha(t_0)$. First, note $\sup_{t>0}f_\alpha(t)\geq\lim_{t\downarrow 0}f_\alpha(t)=1$. By the assumption $\alpha<\alpha_v$ and Lemma \ref{D-f-g}, we obtain 
$\alpha<\alpha_v\leq\frac{1}{E_q}g(t_0)$, which implies $f_\alpha(t_0)<1$. 
Then we see $1\leq\sup_{t>0}f_\alpha(t)=f_\alpha(t_0)<1$, which is a contradiction. 
Thus $\sup_{t>0}f_\alpha(t)$ is not attained.
\end{proof}

\begin{lem}\label{sub-pro1}
Let $1<q<1^*$, $a>0$ and $b>0$. 

\medskip

\noindent
{\rm(i)} Let $a>N(q-1)$. Then there holds $\alpha_v=0$, 
and $D_\alpha$ is attained for $\alpha>0$. 

\medskip

\noindent
{\rm(ii)} Let $a<N(q-1)$. Then there holds $\alpha_v>0$, 
and $D_\alpha$ is attained for $\alpha\geq\alpha_v$, 
while $D_\alpha$ is not attained for $0<\alpha<\alpha_v$. 
\end{lem}

\begin{proof}
(i) Let $a>N(q-1)$. In this case, since $\lim_{t\downarrow 0}g(t)=0$ and $g(t)>0$ for $t>0$, 
by Lemma \ref{D-f-g}, we obtain $\alpha_v=\frac{1}{E_q}\inf_{t>0}g(t)=\frac{1}{E_q}\lim_{t\downarrow 0}g(t)=0$, 
and then Lemma \ref{routine-attain-lem} (i) implies that $D_\alpha$ is attained for $\alpha>0$. 

\medskip

\noindent
(ii) Let $a<N(q-1)$. The conditions $a<N(q-1)$ and $q<1^*$ imply $\lim_{t\downarrow 0}g(t)=\lim_{t\to\infty}g(t)=\infty$. 
Since $g(t)>0$ for $t>0$, there exists $t_0>0$ such that $\inf_{t>0}g(t)=g(t_0)>0$, 
and then Lemma \ref{D-f-g} shows $\alpha_v=\frac{1}{E_q}\inf_{t>0}g(t)=\frac{1}{E_q}g(t_0)>0$. 
By Lemma \ref{routine-attain-lem}, it remains to prove $D_{\alpha_v}$ is attained, 
which is equivalent to $\sup_{t>0}f_{\alpha_v}(t)$ is attained. 
Note that $\alpha_v=\frac{1}{E_q}g(t_0)$ implies $f_{\alpha_v}(t_0)=1$. Recalling $\lim_{t\downarrow 0}f_{\alpha_v}(t)=1$ 
and $\lim_{t\to\infty}f_{\alpha_v}(t)=0$, we can conclude that $\sup_{t>0}f_{\alpha_v}(t)$ is attained. 
\end{proof}

\begin{lem}\label{attain-n(q-1)}
Let $1<q<1^*$, $a=N(q-1)$ and $b>0$.

\medskip

\noindent
{\rm(i)} Let $\frac{2N-1}{2(N-1)}<q<1^*$ and $b\geq b_0:=(q-1)(N-1)-\left(N-(N-1)q\right)>0$. 
Then there holds $\alpha_v=\frac{1}{b E_q}$, and $D_\alpha$ is attained for $\alpha>\alpha_v$, 
while $D_\alpha$ is not attained for $0<\alpha\leq\alpha_v$. 

\medskip

\noindent
{\rm(ii)} Let $\frac{2N-1}{2(N-1)}<q<1^*$ and $b<b_0$. 
Then there holds $0<\alpha_v<\frac{1}{b E_q}$, and $D_\alpha$ is attained for $\alpha\geq\alpha_v$, 
while $D_\alpha$ is not attained for $0<\alpha<\alpha_v$. 

\medskip

\noindent
{\rm(iii)} Let $1<q\leq\frac{2N-1}{2(N-1)}$. Then there holds $\alpha_v=\frac{1}{bE_q}$, 
and $D_\alpha$ is attained for $\alpha>\alpha_v$, 
while $D_\alpha$ is not attained for $0<\alpha\leq\alpha_v$. 
\end{lem}

\begin{proof}
First, note $\lim_{t\downarrow 0}g(t)=\frac{1}{b}$, and then $\alpha_v=\frac{1}{E_q}\inf_{t>0}g(t)\leq\frac{1}{bE_q}$. 

\medskip

\noindent
(i) Let $\frac{2N-1}{2(N-1)}<q<1^*$ and $b\geq b_0$. Define the function $\phi(t)$ for $t>0$ by 
\begin{align*}
\phi(t):=(1+t)^{1+\frac{N-q(N-1)}{b}}-(1+t)^{1-\frac{(q-1)(N-1)}{b}}-\frac{t}{b}. 
\end{align*}
We can compute 
\begin{align*}
&\phi'(t)=\left(1+\frac{N-q(N-1)}{b}\right)(1+t)^{\frac{N-q(N-1)}{b}}
-\left(1-\frac{(q-1)(N-1)}{b}\right)(1+t)^{-\frac{(q-1)(N-1)}{b}}-\frac{1}{b}
\end{align*}
and $\phi''(t)=(1+t)^{-1-\frac{(q-1)(N-1)}{b}}\varphi(t)$, where 
\begin{align*}
\varphi(t):=\frac{N-q(N-1)}{b}\left(1+\frac{N-q(N-1)}{b}\right)(1+t)^{\frac{1}{b}}+\frac{(q-1)(N-1)}{b}\left(1-\frac{(q-1)(N-1)}{b}\right). 
\end{align*}
We see that the condition $b\geq b_0$ implies $\varphi(t)>\varphi(0)\geq 0$ for $t>0$. 
Hence, $\phi'(t)>\phi'(0)=0$ for $t>0$, and then $\phi(t)>\phi(0)=0$ for $t>0$, 
which is equivalent to $g(t)>\frac{1}{b}$ for $t>0$. Therefore, there holds $\alpha_v=\frac{1}{E_q}\inf_{t>0}g(t)\geq\frac{1}{bE_q}$, 
and it follows $\alpha_v=\frac{1}{bE_q}$. It remains to prove that $D_{\alpha_v}$ is not attained. 
Indeed, since $g(t)>\frac{1}{b}$ for $t>0$ is equivalent to $f_{\alpha_v}(t)<1$ for $t>0$, 
we know that $D_{\alpha_v}=\sup_{t>0}f_{\alpha_v}(t)=1$ is not attained. 

\medskip

\noindent
(ii) Let $\frac{2N-1}{2(N-1)}<q<1^*$ and $b<b_0$. In this case, since the condition $b<b_0$ implies $\varphi(0)<0$, 
there exists a unique $t_0>0$ such that $\phi''(t_0)=0$, $\phi''(t)<0$ for $t\in(0,t_0)$ and $\phi''(t)>0$ for $t\in(t_0,\infty)$. 
Then by noting $\phi'(0)=0$ and $\lim_{t\to\infty}\phi'(t)=\infty$, we see that 
there exists a unique $t_1\in(t_0,\infty)$ such that $\phi'(t_1)=0$, 
$\phi'(t)<0$ for $t\in(0,t_1)$ and $\phi'(t)>0$ for $t\in(t_1,\infty)$. 
Similarly, the facts $\phi(0)=0$ and $\lim_{t\to\infty}\phi(t)=\infty$ imply that 
there exists a unique $t_2\in(t_1,\infty)$ such that $\phi(t_2)=0$, 
$\phi(t)<0$ for $t\in(0,t_2)$ and $\phi(t)>0$ for $t\in(t_2,\infty)$. 
Note that $\phi(t)<0$ for $t\in(0,t_2)$ is equivalent to $g(t)<\frac{1}{b}$ for $t\in(0,t_2)$, 
which implies $\alpha_v=\frac{1}{E_q}\inf_{t>0}g(t)<\frac{1}{bE_q}$. 
Moreover, since $\lim_{t\downarrow 0}g(t)=\frac{1}{b}$, $\lim_{t\to\infty}g(t)=\infty$ and $g(t)>0$ for $t>0$, 
we obtain $\alpha_v=\frac{1}{E_q}\inf_{t>0}g(t)>0$. It remains to check that $D_{\alpha_v}$ is attained. 
The signs of $\phi$ from the above observations give $g(t_2)=\frac{1}{b}$, $g(t)<\frac{1}{b}$ for $t\in(0,t_2)$ and $g(t)>\frac{1}{b}$ for $t\in(t_2,\infty)$. 
These facts together with $\lim_{t\downarrow 0}g(t)=\frac{1}{b}$ yield that there exists $t_3\in(0,t_2)$ satisfying $\inf_{t>0}g(t)=g(t_3)$, 
and hence, $\alpha_v=\frac{1}{E_q}\inf_{t>0}g(t)=\frac{g(t_3)}{E_q}$. 
Therefore, we obtain $\alpha_vE_q=g(t_3)$ which is equivalent to $f_{\alpha_v}(t_3)=1=\lim_{t\downarrow 0}f_{\alpha_v}(t)$. 
Recalling $\lim_{t\to\infty}f_{\alpha_v}(t)=0$, we see that $D_{\alpha_v}=\sup_{t>0}f_{\alpha_v}(t)$ is attained. 

\medskip

\noindent
(iii) Let $1<q\leq\frac{2N-1}{2(N-1)}$. In this case, we see $\phi''(t)>0$ for $t>0$, 
and hence, in the same way as in the case (i), we get the desired result. 
\end{proof}

\begin{prop}\label{sub-pro2-est}
Let $1<q<1^*$, $a>0$ and $b>0$.

\medskip

\noindent
{\rm(i)} Let $\frac{2N-1}{2(N-1)}<q<1^*$ and $a=N(q-1)$. Then there hold 
$\lim_{b\downarrow 0}\alpha_v=\infty$ and $\lim_{b\uparrow b_0}\alpha_v=\frac{1}{b_0E_q}$. 

\medskip

\noindent
{\rm(ii)} Let $a<N(q-1)$. Then there hold $\lim_{b\downarrow 0}\alpha_v=\infty$ and $\lim_{b\to\infty}\alpha_v=0$.

\medskip

\noindent
{\rm(iii)} There holds $\lim_{a\downarrow 0}\alpha_v=\infty$. 

\medskip

\noindent
{\rm(iv)} Let $\frac{2N-1}{2(N-1)}<q<1^*$ and $b\geq b_0$, or let $1<q\leq\frac{2N-1}{2(N-1)}$. 
Then there holds $\lim_{a\uparrow N(q-1)}\alpha_v=\frac{1}{bE_q}$. 
\end{prop}

\begin{proof}
(i) Let $\frac{2N-1}{2(N-1)}<q<1^*$, $a=N(q-1)$ and $b<b_0$. 
In the proof of Lemma \ref{attain-n(q-1)} (ii), we proved that $\inf_{t>0}g(t)$ is attained by some $t=t_b\in(0,\infty)$, and then 
\begin{align*}
&\alpha_v=\frac{1}{E_q}\inf_{t>0}g(t)=\frac{g(t_b)}{E_q}
=\frac{\left((1+t_b)^{\frac{1}{b}}-1\right)(1+t_b)^{1-\frac{(q-1)(N-1)}{b}}}{E_qt_b}\\
&=\frac{(1+t_b)^{1+\frac{N-q(N-1)}{b}}-(1+t_b)^{1-\frac{(q-1)(N-1)}{b}}}{E_qt_b}. 
\end{align*}
Noting $b<b_0<(q-1)(N-1)$ and thus $1-\frac{(q-1)(N-1)}{b}<0$, and using the inequality 
$\frac{(1+t)^{1+\frac{N-q(N-1)}{b}}-1}{t}\geq 1+\frac{N-q(N-1)}{b}$ for $t>0$, we see 
\begin{align*}
\alpha_v\geq\frac{(1+t_b)^{1+\frac{N-q(N-1)}{b}}-1}{E_qt_b}\geq\frac{1}{E_q}\left(1+\frac{N-q(N-1)}{b}\right), 
\end{align*}
which implies $\lim_{b\downarrow 0}\alpha_v=\infty$. 
Next, we prove $\lim_{b\uparrow b_0}\alpha_v=\frac{1}{b_0E_q}$. 
To this end, we claim $\lim_{b\uparrow b_0}t_2=0$, which implies $\lim_{b\uparrow b_0}t_3=0$ since $0<t_3<t_2$, 
where the numbers $t_2$ and $t_3$ are the ones introduced in the proof of Lemma \ref{attain-n(q-1)} (ii). 
We write $t_2=t_2(b)$ and $t_3=t_3(b)$ for $b<b_0$. Take any positive sequence $\{b_i\}_{i=1}^\infty$ satisfying 
$b_i\uparrow b_0$ as $i\to\infty$. Recall that for each $i$, $t_2(b_i)$ satisfies 
\begin{align}\label{t2-equation}
\phi(t_2(b_i))
=\left(
1+t_2(b_i)
\right)^{1+\frac{N-q(N-1)}{b_i}}
-\left(1+t_2(b_i)\right)^{1-\frac{(q-1)(N-1)}{b_i}}-\frac{t_2(b_i)}{b_i}=0. 
\end{align}
Since $1-\frac{(q-1)(N-1)}{b_i}\to 1-\frac{(q-1)(N-1)}{b_0}<0$ as $i\to\infty$, 
the equation \eqref{t2-equation} shows that the sequence $\{t_2(b_i)\}_{i=1}^\infty$ is bounded, 
and hence, we may assume that $\lim_{i\to\infty}t_2(b_i)=c_0$ for some $c_0\geq 0$. 
Then letting $i\to\infty$ in \eqref{t2-equation} gives $\phi(c_0)=0$, which implies $c_0=0$ 
since we proved $\phi(t)>0$ for $t>0$ when $b=b_0$ in the proof of Lemma \ref{attain-n(q-1)} (i). 
Therefore, the fact $\lim_{b\uparrow b_0}t_3(b)=0$ has been proved. Now we see 
\begin{align*}
&\liminf_{b\uparrow b_0}\alpha_v
=\frac{1}{E_q}\liminf_{b\uparrow b_0}\left(\inf_{t>0}g(t)\right)
=\frac{1}{E_q}\liminf_{b\uparrow b_0}g(t_3(b))\\
&=\frac{1}{E_q}\liminf_{b\uparrow b_0}\frac{
\left((1+t_3(b))^{\frac{1}{b}}-1\right)(1+t_3(b))^{1-\frac{(q-1)(N-1)}{b}}
}{t_3(b)}
=\frac{1}{E_q}\liminf_{b\uparrow b_0}\frac{
(1+t_3(b))^{\frac{1}{b}}-1
}{t_3(b)}\\
&\geq\frac{1}{E_q}\lim_{b\uparrow b_0}\frac{
(1+t_3(b))^{\frac{1}{b_0}}-1
}{t_3(b)}
=\frac{1}{E_q}\lim_{t\downarrow 0}\frac{(1+t)^{\frac{1}{b_0}}-1}{t}
=\frac{1}{b_0E_q}. 
\end{align*}
On the other hand, since $\inf_{t>0}g(t)\leq\lim_{t\downarrow 0}g(t)=\frac{1}{b}$ for $b<b_0$, 
we obtain $\limsup_{b\uparrow b_0}\alpha_v=\frac{1}{E_q}\limsup_{b\uparrow b_0}(\inf_{t>0}g(t))\leq\frac{1}{b_0E_q}$. 
As a result, the fact $\lim_{b\uparrow b_0}\alpha_v=\frac{1}{b_0E_q}$ has been proved. 

\medskip

\noindent
(ii)  Let $a<N(q-1)$. We first prove $\lim_{b\downarrow 0}\alpha_v=\infty$. 
Since $\lim_{t\downarrow 0}g(t)=\lim_{t\to\infty}g(t)=\infty$, there exists $t_b\in(0,\infty)$ such that $\inf_{t>0}g(t)=g(t_b)>0$. 
Take any positive sequence $\{b_j\}_{j=1}^\infty$ satisfying $b_j\downarrow 0$ as $j\to\infty$. 
First, suppose $\liminf_{j\to\infty}t_{b_j}\in(0,\infty]$. In this case, we may assume $t_{b_j}\geq c$ for $j$ with some $c>0$. 
Then we see for $j$, 
\begin{align*}
&\alpha_v=\frac{g(t_{b_j})}{E_q}
=\frac{
(1+t_{b_j})^{\frac{N(q-1)}{a}+\frac{N-q(N-1)}{b_j}}\left(1-(1+t_{b_j})^{-\frac{1}{b_j}}\right)
}{E_q
t_{b_j}^{\frac{N(q-1)}{a}}
}\\
&\geq\frac{1}{E_q}(1+c)^{\frac{N-q(N-1)}{b_j}}\left(1-(1+c)^{-\frac{1}{b_j}}\right)\to\infty
\end{align*}
as $j\to\infty$. Hence, there holds $\lim_{b\downarrow 0}\alpha_v=\infty$. 
Next, suppose $\liminf_{j\to\infty}t_{b_j}=0$, and we may assume $t_{b_j}\downarrow 0$ as $j\to\infty$. 
Since $b_j<1$ for large $j\in\Bbb N$, there holds 
$(1+t)^{\frac{1}{b_j}}-1\geq t(1+t)^{\frac{1}{b_j}-1}$ for $t>0$. 
By using this inequality and the condition $a<N(q-1)$, we see 
\begin{align*}
&\alpha_v=\frac{g(t_{b_j})}{E_q}
=\frac{
\left((1+t_{b_j})^{\frac{1}{b_j}}-1\right)(1+t_{b_j})^{(q-1)(\frac{N}{a}-\frac{N-1}{b_j})}
}{
E_qt_{b_j}^{\frac{N(q-1)}{a}}
}\\
&\geq\frac{
(1+t_{b_j})^{
\frac{N(q-1)}{a}-1+\frac{N-q(N-1)}{b_j}
}
}{E_qt_{b_j}^{\frac{N(q-1)}{a}-1}}
\geq\frac{1}{E_qt_{b_j}^{\frac{N(q-1)}{a}-1}}\to\infty
\end{align*}
as $j\to\infty$. Then there holds $\lim_{b\downarrow 0}\alpha_v=\infty$. Next, we see 
\begin{align*}
\alpha_v=\frac{1}{E_q}\inf_{t>0}g(t)\leq\frac{g(1)}{E_q}
=\frac{1}{E_q}(2^{\frac{1}{b}}-1)2^{(q-1)(\frac{N}{a}-\frac{N-1}{b})}\to0
\end{align*}
as $b\to\infty$, and thus it follows $\lim_{b\to\infty}\alpha_v=0$. 

\medskip

\noindent
(iii) We prove $\lim_{a\downarrow 0}\alpha_v=\infty$. 
Let $a<N(q-1)$. Then since $\lim_{t\downarrow 0}g(t)=\lim_{t\to\infty}g(t)=\infty$, 
there exists $t_a\in(0,\infty)$ such that $\inf_{t>0}g(t)=g(t_a)>0$. 
Take any positive sequence $\{a_j\}_{j=1}^\infty$ satisfying $N(q-1)>a_j\downarrow 0$ as $j\to\infty$. 
First, suppose $\liminf_{j\to\infty}t_{a_j}\in(0,\infty]$. 
In this case, we may assume $t_{a_j}\geq c$ for $j$ with some $c>0$. 
Then we see for $j$, 
\begin{align}
&\notag\alpha_v=\frac{g(t_{a_j})}{E_q}=
\frac{
(1+t_{a_j})^{\frac{N(q-1)}{a_j}+\frac{N-q(N-1)}{b}}\left(1-(1+t_{a_j})^{-\frac{1}{b}}\right)
}{E_qt_{a_j}^{\frac{N(q-1)}{a_j}}}\\
&\label{ato0-est}\geq\frac{1}{E_q}\left(1+\frac{1}{t_{a_j}}\right)^{\frac{N(q-1)}{a_j}}(1+t_{a_j})^{\frac{N-q(N-1)}{b}}\left(1-(1+c)^{-\frac{1}{b}}\right). 
\end{align}
Suppose $\liminf_{j\to\infty}t_{a_j}=\infty$. Then by \eqref{ato0-est}, we have  
\begin{align*}
\alpha_v\geq\frac{1}{E_q}(1+t_{a_j})^{\frac{N-q(N-1)}{b}}\left(1-(1+c)^{-\frac{1}{b}}\right)\to\infty
\end{align*}
as $j\to\infty$, and hence $\lim_{j\to\infty}\alpha_v=\infty$. 
Suppose $\liminf_{j\to\infty}t_{a_j}\in(0,\infty)$. Then we may assume that $c\leq t_{a_j}\leq\tilde c$ for $j$ with some $\tilde c>c$. 
Hence, by \eqref{ato0-est}, we see 
\begin{align*}
\alpha_v\geq\frac{1}{E_q}\left(1+\frac{1}{\tilde c}\right)^{\frac{N(q-1)}{a_j}}(1+c)^{\frac{N-q(N-1)}{b}}\left(1-(1+c)^{-\frac{1}{b}}\right)\to\infty
\end{align*}
as $j\to\infty$, and hence $\lim_{j\to\infty}\alpha_v=\infty$. 
Next, suppose $\liminf_{j\to\infty}t_{a_j}=0$, and we may assume $t_{a_j}\downarrow 0$ as $j\to\infty$. 
Then we see 
\begin{align*}
\alpha_v=\frac{g(t_{a_j})}{E_q}\geq\frac{1-(1+t_{a_j})^{-\frac{1}{b}}}{E_q t_{a_j}^{\frac{N(q-1)}{a_j}}}\to\infty
\end{align*}
as $j\to\infty$, and hence, $\lim_{j\to\infty}\alpha_v=\infty$. 
As a conclusion, we have proved $\lim_{a\downarrow 0}\alpha_v=\infty$. 

\medskip

\noindent
(iv) Let $\frac{2N-1}{2(N-1)}<q<1^*$ and $b\geq b_0$, or let $1<q\leq\frac{2N-1}{2(N-1)}$. 
We prove $\lim_{a\uparrow N(q-1)}\alpha_v=\frac{1}{bE_q}$. 
For $a>0$, we write $\alpha_v=\alpha_v(a)$ and $g=g_a$. Letting $a<N(q-1)$, we see 
$\alpha_v(a)\leq\frac{g_a(t)}{E_q}$ for $t>0$, and then 
$\limsup_{a\uparrow N(q-1)}\alpha_v(a)\leq\frac{g_{N(q-1)}(t)}{E_q}$ for $t>0$. 
Taking the infimum for $t\in(0,\infty)$ in this relation yields $\limsup_{a\uparrow N(q-1)}\alpha_v(a)\leq\alpha_v(N(q-1))=\frac{1}{bE_q}$, 
where we have used Lemma \ref{attain-n(q-1)} (i) and (iii). 
Next, let $a<N(q-1)$, and let $t_a\in(0,\infty)$ be a point satisfying $\inf_{t>0}g_a(t)=g_a(t_a)$, 
and hence, $\alpha_v(a)=\frac{g_a(t_a)}{E_q}$. 
Take any positive sequence $\{a_j\}_{j=1}^\infty$ satisfying 
$a_j\uparrow N(q-1)$ as $j\to\infty$. First, suppose $\overline\lim_{j\to\infty}t_{a_j}\in(0,\infty)$, and then we may assume 
$t_{a_j}\to t_0\in(0,\infty)$ as $j\to\infty$. We see 
\begin{align*}
\lim_{j\to\infty}\alpha_v(a_j)=\frac{1}{E_q}\lim_{j\to\infty}g_{a_j}(t_{a_j})=\frac{g_{N(q-1)}(t_0)}{E_q}
\geq\frac{1}{E_q}\inf_{t>0}g_{N(q-1)}g(t)=\alpha_v(N(q-1))=\frac{1}{bE_q}, 
\end{align*}
where we have used Lemma \ref{attain-n(q-1)} (i) and (iii). 
On the other hand, since we have already proved $\lim_{j\to\infty}\alpha_v(a_j)\leq\frac{1}{bE_q}$, 
we obtain $g_{N(q-1)}(t_0)=\frac{1}{b}$. However, this is impossible since we observed $g_{N(q-1)}(t)>\frac{1}{b}$ for $t>0$ 
in the proof of Lemma \ref{attain-n(q-1)} (i) and (iii). Next, suppose $\overline\lim_{j\to\infty}t_{a_j}=\infty$, and then we may assume 
$t_{a_j}\to\infty$ as $j\to\infty$. We see for $j$, 
\begin{align*}
&\alpha_v(a_j)=\frac{g_{a_j}(t_{a_j})}{E_q}
=\frac{(1+t_{a_j})^{\frac{N(q-1)}{a_j}+\frac{N-q(N-1)}{b}}\left(1-(1+t_{a_j})^{-\frac{1}{b}}\right)}{
E_qt_{a_j}^{\frac{N(q-1)}{a_j}}
}\\&\geq\frac{1}{E_q}(1+t_{a_j})^{\frac{N-q(N-1)}{b}}\left(1-(1+t_{a_j})^{-\frac{1}{b}}\right)\to\infty
\end{align*}
as $j\to\infty$, and hence, $\lim_{j\to\infty}\alpha_v(a_j)=\infty$, 
which is a contradiction since we have already proved $\limsup_{j\to\infty}\alpha_v(a_j)\leq\frac{1}{bE_q}$. 
As a result, it holds $\lim_{j\to\infty}t_{a_j}=0$, and hence, $\lim_{a\uparrow N(q-1)}t_a=0$. 
Then we see 
\begin{align*}
&\liminf_{a\uparrow N(q-1)}\alpha_v(a)=\frac{1}{E_q}\liminf_{a\uparrow N(q-1)}g_a(t_a)
=\frac{1}{E_q}\liminf_{a\uparrow N(q-1)}\frac{
\left((1+t_a)^{\frac{1}{b}}-1\right)(1+t_a)^{(q-1)(\frac{N}{a}-\frac{N-1}{b})}
}{t_a^{\frac{N(q-1)}{a}}}\\
&\geq\frac{1}{E_q}\liminf_{a\uparrow N(q-1)}\frac{(1+t_a)^{\frac{1}{b}}-1}{t_a^{\frac{N(q-1)}{a}}}
\geq\frac{1}{E_q}\lim_{a\uparrow N(q-1)}\frac{(1+t_a)^{\frac{1}{b}}-1}{t_a}=\frac{1}{bE_q}. 
\end{align*}
As a conclusion, we have $\lim_{a\uparrow N(q-1)}\alpha_v(a)=\frac{1}{bE_q}$. 
\end{proof}

\begin{proof}[{\rm \bf Proof of Theorems \ref{thm1}-\ref{thm2}}]
Gathering up Lemmas \ref{sub-pro1}-\ref{attain-n(q-1)} and Proposition \ref{sub-pro2-est}, we have the results stated in Theorems \ref{thm1}-\ref{thm2}. 
\end{proof}

\section{Proof of Theorems \ref{thm3}-\ref{thm4}}
\noindent

In this section, we shall prove Theorems \ref{thm3}-\ref{thm4}. 
We start from the following lemma : 

\begin{lem}\label{alpha^*-h}
Let $\alpha>0$, $a>0$ and $b>0$. Then there hold $D_\alpha=\sup_{t>0}f_\alpha(t)$, 
$\alpha_v=\frac{1}{E_{1^*}}\inf_{t>0}g(t)$ and $\alpha_c=\frac{1}{E_{1^*}}\sup_{t>0}h(t)$ where for $t>0$, 
\begin{align*}
\begin{cases}
&f_\alpha(t):=\displaystyle\frac{(1+t)^{\frac{1^*}{a}-\frac{1}{b}}+\alpha E_{1^*}t^{\frac{1^*}{a}}}{(1+t)^{\frac{1^*}{a}}},\\
&g(t):=\displaystyle\frac{(1+t)^{\frac{1^*}{a}}-(1+t)^{\frac{1^*}{a}-\frac{1}{b}}}{t^{\frac{1^*}{a}}},\\
&h(t):=\displaystyle\frac{(1+t)^{\frac{1^*}{a}-\frac{1}{b}}}{(1+t)^{\frac{1^*}{a}}-t^{\frac{1^*}{a}}}. 
\end{cases}
\end{align*}
\end{lem}

\begin{proof}
The former two equalities are obtained by putting $q=1^*$ in Lemma \ref{D-f-g}. 
Hence, we consider $\alpha_c$. For $u\in BV$ with $\|u\|_{TV}^a++\|u\|_{1}^b=1$, we see 
\begin{align*}
&\frac{\|u\|_{1}}{E_{1^*}-\|u\|_{1^*}^{1^*}}\leq\frac{\|u\|_{1}}{E_{1^*}(1-\|u\|_{TV}^{1^*})}
=\frac{
\left(
1+\frac{\|u\|_{TV}^a}{\|u\|_{1}^b}
\right)^{\frac{1^*}{a}-\frac{1}{b}}
}{
E_{1^*}\left(
\left(
1+\frac{\|u\|_{TV}^a}{\|u\|_{1}^b}
\right)^{\frac{1^*}{a}}-\left(\frac{\|u\|_{TV}^a}{\|u\|_{1}^b}\right)^{\frac{1^*}{a}}
\right)
}\\
&=\frac{1}{E_{1^*}}h\left(\frac{\|u\|_{TV}^a}{\|u\|_{1}^b}\right)
\leq\frac{1}{E_{1^*}}\sup_{t>0}h(t), 
\end{align*}
which shows $\alpha_c\leq\frac{1}{E_{1^*}}\sup_{t>0}h(t)$. 
On the other hand, let $v\in BV\setminus\{0\}$ be a maximizer of $E_{1^*}$. 
For $\lambda>0$, define $v_\lambda(x):=K\lambda v(\lambda^{\frac{1}{N}}x)$, 
where $K=K(\lambda)>0$ is uniquely determined by 
\begin{align*}
\|v_\lambda\|_{TV}^a+\|v_\lambda\|_{1}^b=K^a\lambda^{\frac{a}{N}}\|v\|_{TV}^a+K^b\|v\|_{1}^b=1. 
\end{align*}  
Then for $\lambda>0$, we observe 
\begin{align*}
&\alpha_c\geq\frac{\|v_\lambda\|_{1}}{E_{1^*}-\|v_\lambda\|_{1^*}^{1^*}}
=\frac{\|v_\lambda\|_{1}}{E_{1^*}(1-\|v_\lambda\|_{TV}^{1^*})}
=\frac{1}{E_{1^*}}h\left(\frac{\|v_\lambda\|_{TV}^a}{\|v_\lambda\|_{1}^b}\right)\\
&=\frac{1}{E_{1^*}}h\left(K^{a-b}\lambda^{\frac{a}{N}}\frac{\|v\|_{TV}^a}{\|v\|_{1}^b}\right)
=\frac{1}{E_{1^*}}h\left(\frac{1}{K^b\|v\|_{1}^b}-1\right). 
\end{align*}
Since $K=K(\lambda)$ is a continuous function on $(0,\infty)$ 
satisfying $K<\frac{1}{\|v\|_{1}}$ for $\lambda>0$, 
$\lim_{\lambda\downarrow 0}K=\frac{1}{\|v\|_{1}}$ and $\lim_{\lambda\to\infty}K=0$, 
we obtain 
\begin{align*}
&\alpha_c\geq\frac{1}{E_{1^*}}\sup_{\lambda>0}h\left(\frac{1}{K^b\|v\|_{1}^b}-1\right)
=\frac{1}{E_{1^*}}\sup_{t>0}h(t). 
\end{align*}
Hence, the proof of Lemma \ref{alpha^*-h} is complete. 
\end{proof}

\begin{lem}\label{vani-conce-ge}Let $a>0$ and $b>0$. 

\medskip

\noindent
{\rm(i)} Assume $\alpha_c<\infty$ and let $\alpha>\alpha_c$. 
Then $D_\alpha$ is not attained. 

\medskip

\noindent
{\rm(ii)} Assume $\alpha_v>0$ and let $\alpha<\alpha_v$. 
Then $D_\alpha$ is not attained. 

\medskip

\noindent
{\rm(iii)} Assume $\alpha_v<\alpha_c$ and let $\alpha_v<\alpha<\alpha_c$. 
Then $D_\alpha$ is attained. 
\end{lem}

\begin{proof}
By Theorem \ref{thm5}, we see that $D_\alpha$ is attained if and only if $\sup_{t>0}f_\alpha(t)$ is attained. 

\medskip

\noindent
(i) Assume $\alpha_c<\infty$ and let $\alpha>\alpha_c$. 
By contradiction, assume that there exists $t_0>0$ such that 
$\sup_{t>0}f_\alpha(t)=f_\alpha(t_0)$. First, note $\sup_{t>0}f_\alpha(t)\geq\lim_{t\to\infty}f_\alpha(t)=\alpha E_{1^*}$. By Lemma \ref{alpha^*-h} and the assumption $\alpha>\alpha_c$, 
we obtain $\alpha>\alpha_c\geq\frac{1}{E_{1^*}}h(t_0)$, which implies $f_\alpha(t_0)<\alpha E_{1^*}$. 
Then we see $\alpha E_{1^*}\leq\sup_{t>0}f_\alpha(t)=f_\alpha(t_0)<\alpha E_{1^*}$, 
which is a contradiction. Thus $\sup_{t>0}f_\alpha(t)$ is not attained. 

\medskip

\noindent
(ii) Assume $\alpha_v>0$ and let $\alpha<\alpha_v$. 
By contradiction, assume that there exists $t_0>0$ such that $\sup_{t>0}f_\alpha(t)=f_\alpha(t_0)$. 
First, note $\sup_{t>0}f_\alpha(t)\geq\lim_{t\downarrow 0}f_\alpha(t)=1$. 
By Lemma \ref{alpha^*-h} and the assumption $\alpha<\alpha_v$, we obtain 
$\alpha<\alpha_v\leq\frac{1}{E_{1^*}}g(t_0)$, which implies $f_\alpha(t_0)<1$. 
Then we see $1\leq\sup_{t>0}f_\alpha(t)=f_\alpha(t_0)<1$, which is a contradiction. Thus $\sup_{t>0}f_\alpha(t)$ is not attained. 

\medskip

\noindent
(iii) Assume $\alpha_v<\alpha_c$ and let $\alpha_v<\alpha<\alpha_c$. 
First, note that $\lim_{t\downarrow 0}f_\alpha(t)=1$ and $\lim_{t\to\infty}f_\alpha(t)=\alpha E_{1^*}$. 
By the assumption $\alpha>\alpha_v$, there exists $t_0>0$ such that 
$\alpha>\frac{1}{E_{1^*}}g(t_0)$, which implies $f_\alpha(t_0)>1$. 
On the other hand, by the assumption $\alpha<\alpha_c$, 
there exists $t_1>0$ such that $\alpha<\frac{1}{E_{1^*}}h(t_1)$, which implies $f_\alpha(t_1)>\alpha E_{1^*}$. As a result, $\sup_{t>0}f_\alpha(t)$ is attained. 
\end{proof}

\begin{lem}\label{a>1^*-ests}
Let $a>1^*$ and $b>0$. Then there hold 
\begin{align*}
\alpha_v=0\text{ \,and \,}
\alpha_c=\begin{cases}
&\infty\text{ \,when \,}b>1,\\
&\frac{1}{E_{1^*}}<\alpha_c<\infty\text{ \,when \,}b\leq1.
\end{cases}
\end{align*} In particular, there hold 
$\lim_{b\downarrow 0}\alpha_c=\frac{1}{E_{1^*}}$, 
$\lim_{b\uparrow 1}\alpha_c=\frac{a}{1^*E_{1^*}}$, $\alpha_c=\frac{a}{1^*E_{1^*}}$ when $b=1$, 
$\lim_{a\downarrow 1^*}\alpha_c=\frac{1}{E_{1^*}}$ when $b\leq1$ and $\lim_{a\to\infty}\alpha_c=\infty$ when $b\leq1$. 
Moreover, 
\begin{align*}
D_\alpha\text{ is attained for }
\begin{cases}
&\alpha>0\text{ \,when \,}b>1,\\
&0<\alpha<\alpha_c\text{ \,when \,}b=1,\\
&0<\alpha\leq\alpha_c\text{ \,when \,}b<1,
\end{cases}
\end{align*}
while 
\begin{align*}
D_\alpha\text{ is not attained for }
\begin{cases}
&\alpha\geq\alpha_c\text{ \,when \,}b=1,\\
&\alpha>\alpha_c\text{ \,when \,}b<1. 
\end{cases}
\end{align*}
\end{lem}

\begin{proof}
Let $a>1^*$. First, we can compute 
\begin{align*}
\lim_{t\downarrow 0}g(t)=\lim_{t\downarrow 0}\frac{(1+t)^{\frac{1}{b}}-1}{t^{\frac{1^*}{a}}}
=\lim_{t\downarrow 0}\frac{a}{b1^*}t^{1-\frac{1^*}{a}}(1+t)^{\frac{1}{b}-1}=0, 
\end{align*}
and then noting $g(t)>0$ for $t>0$, we have 
$\alpha_v=\frac{1}{E_{1^*}}\inf_{t>0}g(t)=\frac{1}{E_{1^*}}\lim_{t\downarrow 0}g(t)=0$. 
Next, we see $\lim_{t\downarrow 0}h(t)=1$ and 
\begin{align*}
\lim_{t\to\infty}h(t)
=\lim_{t\to\infty}\frac{(1+t)^{-\frac{1}{b}}}{1-(\frac{t}{1+t})^{\frac{1^*}{a}}}
=\lim_{t\to\infty}\frac{a}{b1^*}\frac{(1+t)^{-\frac{1}{b}+1}}{(\frac{t}{1+t})^{\frac{1^*}{a}-1}}
=\begin{cases}
&\infty\text{ \,when \,}b>1,\\
&\frac{a}{1^*}\text{ \,when \,}b=1,\\
&0\text{ \,when \,}b<1. 
\end{cases}
\end{align*}
We distinguish between three cases. 
First, let $b>1$. Then it holds $\alpha_c=\frac{1}{E_{1^*}}\sup_{t>0}h(t)=
\frac{1}{E_{1^*}}\lim_{t\to\infty}h(t)=\infty$, 
and then by Lemma \ref{vani-conce-ge} together with $\alpha_v=0$, 
$D_\alpha$ is attained for $\alpha>0$. 
Next, let $b=1$. In this case, it follows $\lim_{t\to\infty}h(t)=\frac{a}{1^*}$. 
By a direct computation, we obtain for $t>0$, 
\begin{align*}
h'(t)=\frac{t^{\frac{1^*}{a}}(1+t)^{\frac{1^*}{a}-2}}{\left((1+t)^{\frac{1^*}{a}}-t^{\frac{1^*}{a}}\right)^2}\tilde h(t)\text{ \,and \,}
\tilde h'(t)=\frac{1^*}{at^2}\left(\left(\frac{t}{1+t}\right)^{1-\frac{1^*}{a}}-1\right)<0, 
\end{align*}
where $\tilde h(t):=1+\frac{1^*}{at}-(\frac{1+t}{t})^{\frac{1^*}{a}}$. 
Since $\lim_{t\to\infty}\tilde h(t)=0$, we observe $\tilde h(t)>0$ for $t>0$, 
which implies $h'(t)>0$ for $t>0$. Summing-up, we have $\lim_{t\downarrow 0}h(t)=1$, 
$\lim_{t\to\infty}h(t)=\frac{a}{1^*}>1$ and $h'(t)>0$ for $t>0$, which show 
$\alpha_c=\frac{1}{E_{1^*}}\sup_{t>0}h(t)=\frac{1}{E_{1^*}}\lim_{t\to\infty}h(t)=\frac{a}{1^*E_{1^*}}$ and $1<h(t)<\frac{a}{1^*}$ for $t>0$. 
Thus by Lemma \ref{vani-conce-ge}, $D_\alpha$ is attained for $0<\alpha<\alpha_c(=\frac{a}{1^*E_{1^*}})$, while $D_\alpha$ is not attained for $\alpha>\alpha_c$. 
Furthermore, the relation $h(t)<\frac{a}{1^*}=\alpha_c E_{1^*}$ for $t>0$ 
implies $f_{\alpha_c}(t)<\alpha_c E_{1^*}=\lim_{s\to\infty}f_{\alpha_c}(s)\leq\sup_{s>0}f_{\alpha_c}(s)$ for $t>0$. Hence, $\sup_{t>0}f_{\alpha_c}(t)$ is not attained, 
which is equivalent to the non-attainability of $D_{\alpha_c}$ by Theorem \ref{thm5}. Next, let $b<1$. By a direct computation, we have for $t>0$, 
\begin{align*}
h'(t)=\frac{t^{\frac{1^*}{a}}(1+t)^{\frac{1^*}{a}-\frac{1}{b}-1}}{b\left((1+t)^{\frac{1^*}{a}}-t^{\frac{1^*}{a}}\right)^2}\tilde h(t)\text{ \,and \,}\tilde h'(t)=\frac{1^*}{at^2}\left(
\left(\frac{t}{1+t}\right)^{1-\frac{1^*}{a}}-b
\right), 
\end{align*}
where $\tilde h(t):=1+\frac{b1^*}{at}-(\frac{1+t}{t})^{\frac{1^*}{a}}$. Then we obtain 
\begin{align*}
\tilde h'(t)\begin{cases}
&<0\text{ \,for \,}0<t<t_0,\\
&=0\text{ \,for \,}t=t_0,\\
&>0\text{ \,for \,}t>t_0, 
\end{cases}
\end{align*} 
where $t_0:=\frac{b^{\frac{a}{a-1^*}}}{1-b^{\frac{a}{a-1^*}}}>0$. 
Since $\lim_{t\to\infty}\tilde h(t)=0$ and 
$\tilde h(t)=\frac{1}{t}\left(\frac{b1^*}{a}+t-t^{1-\frac{1^*}{a}}(1+t)^{\frac{1^*}{a}}\right)\to\infty$ as $t\downarrow 0$, there exists $t_1\in(0,t_0)$ such that 
\begin{align*}
\tilde h(t)
\begin{cases}
&>0\text{ \,for \,}0<t<t_1,\\
&=0\text{ \,for \,}t=t_1,\\
&<0\text{ \,for \,}t>t_1,
\end{cases}
\end{align*}
which implies 
\begin{align*}
h'(t)
\begin{cases}
&>0\text{ \,for \,}0<t<t_1,\\
&=0\text{ \,for \,}t=t_1,\\
&<0\text{ \,for \,}t>t_1. 
\end{cases}
\end{align*}
This fact together with $\lim_{t\downarrow 0}h(t)=1$ and $\lim_{t\to\infty}h(t)=0$ shows 
$\alpha_c=\frac{1}{E_{1^*}}\sup_{t>0}h(t)=\frac{1}{E_{1^*}}h(t_1)$, 
and then it follows $\frac{1}{E_{1^*}}<\alpha_c<\infty$. 
Thus by Lemma \ref{vani-conce-ge}, $D_\alpha$ is attained for $0<\alpha<\alpha_c$, while $D_\alpha$ is not attained for $\alpha>\alpha_c$. 
Furthermore, note that $\alpha_c=\frac{1}{E_{1^*}}h(t_1)$ 
is equivalent to $f_{\alpha_c}(t_1)=\alpha_c E_{1^*}$. 
This fact together with $\lim_{t\downarrow 0}f_{\alpha_c}(t)=1$ 
and $\lim_{t\to\infty}f_{\alpha_c}(t)=\alpha_c E_{1^*}=h(t_1)>1$, 
we can conclude that $\sup_{t>0}f_{\alpha_c}(t)$ is attained, 
and hence, $D_{\alpha_c}$ is attained by Theorem \ref{thm5}. 
It remains to show the asymptotic behaviors of $\alpha_c$ on $a$ and $b$. 
First, we prove $\lim_{b\downarrow 0}\alpha_c=\frac{1}{E_{1^*}}$. 
Since $0<t_1<t_0\to0$ as $b\downarrow 0$, we have $t_1\to 0$ as $b\downarrow 0$, 
and then we see $\lim_{b\downarrow 0}h(t_1)=1$, which implies 
$\lim_{b\downarrow 0}\alpha_c=\frac{1}{E_{1^*}}\lim_{b\downarrow 0}h(t_1)=\frac{1}{E_{1^*}}$. 
Next, we prove $\lim_{b\uparrow 1}\alpha_c=\frac{a}{1^*E_{1^*}}$. 
We write $t_1=t_1(b)$ for $0<b<1$. First, we claim $\lim_{b\uparrow 1}t_1(b)=\infty$. 
On the contrary, assume $\underline\lim_{b\uparrow 1}t_1(b)<\infty$. 
Then we can pick up a sequence $\{b_j\}_{j=1}^\infty\subset(0,1)$ satisfying 
$b_j\uparrow 1$ as $j\to\infty$ and $\lim_{j\to\infty}t_1(b_j)=\overline t_1\in[0,\infty)$. 
Recall that $t_1(b_j)$ satisfies $\tilde h(t_1(b_j))=0$, which implies 
\begin{align}\label{b_j-eq}
\frac{b_j1^*}{a}+t_1(b_j)-t_1(b_j)^{1-\frac{1^*}{a}}\left(1+t_1(b_j)\right)^{\frac{1^*}{a}}=0 
\end{align}
for each $j$. Letting $j\to\infty$ in \eqref{b_j-eq}, we obtain 
\begin{align*}
\frac{1^*}{a}+\overline t_1-\overline t_1^{1-\frac{1^*}{a}}\left(1+\overline t_1\right)^{\frac{1^*}{a}}=0,
\end{align*}
which shows that $\overline t_1>0$ is a solution of $\tilde h(t)=0$ for $t>0$ with $b=1$. 
On the other hand, in the same way as above, we see that $\tilde h(t)$ 
for $t>0$ with $b=1$ satisfies $\lim_{t\downarrow 0}\tilde h(t)=\infty$, 
$\lim_{t\to\infty}\tilde h(t)=0$ and $\tilde h'(t)<0$ for $t>0$, 
and hence, it holds $\tilde h(t)>0$ for $t>0$, 
which is a contradiction to $\tilde h(\overline t_1)=0$. 
As a result, we obtain $\underline\lim_{b\uparrow 1}t_1(b)=\infty$, 
which is equivalent to $\lim_{b\uparrow 1}t_1(b)=\infty$. 
Now we compute $\lim_{b\uparrow 1}h(t_1)$. 
Since $t_1$ satisfies $\tilde h(t_1)=0$, we have 
$(1+t_1)^{\frac{1^*}{a}}=t_1^{-\frac{a-1^*}{a}}\left(\frac{b1^*}{a}+t_1\right)$. 
Plugging this relation to $h(t_1)$, we obtain 
\begin{align*}
h(t_1)
=\frac{a}{b1^*}\left(\frac{t_1}{t_1+\frac{b1^*}{a}}\right)^{\frac{a-1^*}{1^*}}
\frac{
t_1^{\frac{(a-1^*)(1-b)}{b1^*}}
}{
\left(t_1+\frac{b1^*}{a}\right)^{\frac{a(1-b)}{b1^*}}
}. 
\end{align*} 
Since $t_1(b)\to\infty$ as $b\uparrow 1$, in order to prove $\lim_{b\uparrow 1}h(t_1)=\frac{a}{1^*}$, it is enough to show that $\lim_{b\uparrow 1}t_1^{1-b}=1$. 
Recalling $0<t_1<t_0$, we have $(1-b)\log t_1\leq(1-b)\log t_0\to 0$ as $b\uparrow 1$, 
which shows $\lim_{b\uparrow 1}(1-b)\log t_1(b)=0$, 
and hence, it holds $\lim_{b\uparrow 1}t_1^{1-b}=1$. 
As a conclusion, we obtain $\lim_{b\uparrow 1}h(t_1)=\frac{a}{1^*}$, 
which gives $\lim_{b\uparrow 1}\alpha_c=\frac{a}{1^*E_{1^*}}$. 
Next, we prove $\lim_{a\to\infty}\alpha_c=\infty$ when $b\leq 1$. 
We may assume $b<1$ since we have already proved $\alpha_c=\frac{a}{1^*E_{1^*}}$ when $a>1^*$ and $b=1$. 
Noting $t_0\to\frac{b}{1-b}$ as $a\to\infty$, we see $h(t_0)\to\infty$ as $a\to\infty$. 
Then since $t_1(<t_0)$ is the maximum point of $h(t)$ for $t>0$, we see $h(t_1)>h(t_0)\to\infty$ as $a\to\infty$, 
which shows $\lim_{a\to\infty}h(t_1)=\infty$, and hence, it holds $\lim_{a\to\infty}\alpha_c=\frac{1}{E_{1^*}}\lim_{a\to\infty}h(t_1)=\infty$. 
Next we show $\lim_{a\downarrow 1^*}\alpha_c=\frac{1}{E_{1^*}}$ when $b\leq 1$. 
In the same reason as above, we may assume $b<1$. 
Since $b<1$, we see $0<t_1<t_0\to 0$ as $a\downarrow 1^*$, and hence, it holds $\lim_{a\downarrow 1^*}t_1=0$. 
Then we have $\lim_{a\downarrow 1^*}h(t_1)=1$, which is equivalent to $\lim_{a\downarrow 1^*}\alpha_c=\frac{1}{E_{1^*}}$. 
The proof of Lemma \ref{a>1^*-ests} is complete. 
\end{proof}

\begin{lem}
Let $a=1^*$ and $b>0$. 

\medskip

\noindent
{\rm(i)} Let $b>1$. Then there hold $\alpha_v=\frac{1}{bE_{1^*}}$ and $\alpha_c=\infty$, 
and $D_\alpha$ is attained for $\alpha>\alpha_v$, 
while $D_\alpha$ is not attained for $0<\alpha\leq\alpha_v$. 

\medskip

\noindent
{\rm(ii)} Let $b=1$. Then there holds $\alpha_v=\alpha_c=\frac{1}{E_{1^*}}$, 
and $D_\alpha$ is not attained for $\alpha\ne \alpha_v$, 
while $D_{\alpha_v}$ is attained. 

\medskip

\noindent
{\rm(iii)} Let $b<1$. Then there holds $\alpha_v=\alpha_c=\frac{1}{E_{1^*}}$, 
and $D_\alpha$ is not attained for $\alpha>0$. 
\end{lem}

\begin{proof}
(i) Let $a=1^*$ and $b>1$. Since $g(t)=\frac{(1+t)\left(1-(1+t)^{-\frac{1}{b}}\right)}{t}$ 
for $t>0$, we see $\lim_{t\downarrow 0}g(t)=\frac{1}{b}$ and $\lim_{t\to\infty}g(t)=1$. 
We can compute for $t>0$, $g'(t)=t^{-2}(1+t)^{-\frac{1}{b}}\tilde g(t)$, 
where $\tilde g(t):=1+\frac{t}{b}-(1+t)^{\frac{1}{b}}$, 
and we obtain for $t>0$, 
$\tilde g'(t)=\frac{1}{b}-\frac{1}{b}(1+t)^{\frac{1}{b}-1}>0$ since $b>1$. 
Then noting $\lim_{t\downarrow 0}\tilde g(t)=0$, we have $\tilde g(t)>0$ for $t>0$, 
which implies $g'(t)>0$ for $t>0$. 
Here, recalling $\lim_{t\downarrow 0}g(t)=\frac{1}{b}<1$ and $\lim_{t\to\infty}g(t)=1$, 
we obtain $\alpha_v=\frac{1}{E_{1^*}}\inf_{t>0}g(t)=\frac{1}{bE_{1^*}}$. 
On the other hand, since $h(t)=(1+t)^{1-\frac{1}{b}}$ for $t>0$, we obtain 
$\alpha_c=\frac{1}{E_{1^*}}\sup_{t>0}h(t)=\infty$ since $b>1$. 
Thus by Lemma \ref{vani-conce-ge}, $D_\alpha$ is attained for $\alpha>\alpha_v$, while $D_\alpha$ is not attained for $0<\alpha<\alpha_v$. 
Next, we consider the case $\alpha=\alpha_v$. 
Note $\alpha_v E_{1^*}=\frac{1}{b}<g(t)$ for $t>0$, which implies 
$f_{\alpha_v}(t)<1=\lim_{s\downarrow 0}f_{\alpha_v}(s)\leq\sup_{s>0}f_{\alpha_v}(s)$ 
for $t>0$. Hence, $\sup_{t>0}f_{\alpha_v}(t)$ is not attained, 
which is equivalent to the non-attainability of $D_{\alpha_v}$ by Theorem \ref{thm5}. 

\medskip

\noindent
(ii) Let $a=1^*$ and $b=1$. 
In this case, since $g(t)=1$ for $t>0$, it follows 
$\alpha_v=\frac{1}{E_{1^*}}\inf_{t>0}g(t)=\frac{1}{E_{1^*}}$. 
On the other hand, since $h(t)=1$ for $t>0$, it follows 
$\alpha_c=\frac{1}{E_{1^*}}\sup_{t>0}h(t)=\frac{1}{E_{1^*}}$. 
Thus there holds $\alpha_v=\alpha_c=\frac{1}{E_{1^*}}$, 
and by Lemma \ref{vani-conce-ge}, $D_\alpha$ is not attained for $\alpha\ne\alpha_v(=\alpha_c)$. Next, we consider the case $\alpha=\alpha_v$. 
In this case, we see $f_{\alpha_v}(t)=1$ for $t>0$, 
and hence, $\sup_{t>0}f_{\alpha_v}(t)$ is attained, which is equivalent 
to the attainability of $D_{\alpha_v}$ by Theorem \ref{thm5}. 

\medskip

\noindent
(iii) Let $a=1^*$ and $b<1$. First, recall $\lim_{t\downarrow 0}g(t)=\frac{1}{b}$ 
and $\lim_{t\to\infty}g(t)=1$. In the same way as in the case (i), we see 
$g'(t)=t^{-2}(1+t)^{-\frac{1}{b}}\tilde g(t)$ 
with $\tilde g(t):=1+\frac{t}{b}-(1+t)^{\frac{1}{b}}$ for $t>0$. 
Then we obtain $\tilde g'(t)=\frac{1}{b}-\frac{1}{b}(1+t)^{\frac{1}{b}-1}<0$ 
for $t>0$ since $b<1$. Thus noting $\lim_{t\downarrow 0}\tilde g(t)=0$, we have 
$\tilde g(t)<0$ for $t>0$, which implies $g'(t)<0$ for $t>0$. 
Since $\lim_{t\downarrow 0}g(t)=\frac{1}{b}>1$ and $\lim_{t\to\infty}g(t)=1$, 
it follows $\alpha_v=\frac{1}{E_{1^*}}\inf_{t>0}g(t)=\frac{1}{E_{1^*}}$. 
On the other hand, since $h(t)=(1+t)^{1-\frac{1}{b}}$ for $t>0$, 
it follows $\alpha_c=\frac{1}{E_{1^*}}\sup_{t>0}h(t)=\frac{1}{E_{1^*}}$. 
Hence, we obtain $\alpha_v=\alpha_c=\frac{1}{E_{1^*}}$, 
and then by Lemma \ref{vani-conce-ge}, $D_\alpha$ is not attained for $\alpha\ne\alpha_v(=\alpha_c)$. Next, we consider the case $\alpha=\alpha_v$. 
Note $\alpha_v E_{1^*}=1<g(t)$ for $t>0$, which implies $f_{\alpha_v}(t)<1=\lim_{s\downarrow 0}f_{\alpha_v}(s)\leq\sup_{s>0}f_{\alpha_v}(s)$ for $t>0$. Hence, $\sup_{t>0}f_{\alpha_v}(t)$ 
is not attained, which is equivalent to the non-attainability of $D_{\alpha_v}$ by Theorem \ref{thm5}.
\end{proof}

\begin{lem}\label{last-lem-cri}
Let $a<1^*$ and $b>0$. 

\medskip

\noindent
{\rm(i)} Let $b>1$. Then there hold $0<\alpha_v<\frac{1}{E_{1^*}}$ and $\alpha_c=\infty$, 
and $D_\alpha$ is attained for $\alpha\geq\alpha_v$, 
while $D_\alpha$ is not attained for $0<\alpha<\alpha_v$. 
Moreover, there hold $\lim_{b\downarrow 1}\alpha_v=\frac{1}{E_{1^*}}$, $\lim_{b\to\infty}\alpha_v=0$, 
$\lim_{a\downarrow 0}\alpha_v=\frac{1}{E_{1^*}}$ and $\lim_{a\uparrow 1^*}\alpha_v=\frac{1}{b E_{1^*}}$. 

\medskip

\noindent
{\rm(ii)} Let $b\leq1$. Then there holds $\alpha_v=\alpha_c=\frac{1}{E_{1^*}}$, 
and $D_\alpha$ is not attained for $\alpha>0$. 
\end{lem}

\begin{proof}
(i) Let $a<1^*$ and $b>1$. First, we see $\lim_{t\downarrow 0}g(t)=\infty$ and $\lim_{t\to\infty}g(t)=1$. By a direct computation, we have for $t>0$, 
$g'(t)=t^{-\frac{1^*}{a}-1}(1+t)^{\frac{1^*}{a}-\frac{1}{b}-1}\tilde g(t)$, 
where $\tilde g(t):=\frac{t}{b}+\frac{1^*}{a}\left(1-(1+t)^{\frac{1}{b}}\right)$, 
and $\tilde g'(t)=\frac{1}{b}-\frac{1^*}{ab}(1+t)^{\frac{1}{b}-1}$. 
Then we observe 
\begin{align*}
\tilde g'(t)\begin{cases}
&<0\text{ \,for \,}0<t<t_0,\\
&=0\text{ \,for \,}t=t_0,\\
&>0\text{ \,for \,}t>t_0, 
\end{cases}
\end{align*}
where $t_0:=(\frac{1^*}{a})^{\frac{b}{b-1}}-1>0$. 
Note 
\begin{align*}
\lim_{t\downarrow 0}\tilde g(t)=0\text{ \,and \,}
\lim_{t\to\infty}\tilde g(t)=\lim_{t\to\infty}t\left(\frac{1}{b}+\frac{1^*}{at}-\frac{1^*(1+t)^{\frac{1}{b}}}{at}\right)=\infty
\end{align*}
since $b>1$. Hence, there exists $t_1>t_0$ such that 
\begin{align*}
\tilde g(t)\begin{cases}
&<0\text{ \,for \,}0<t<t_1,\\
&=0\text{ \,for \,}t=t_1,\\
&>0\text{ \,for \,}t>t_1,
\end{cases}
\end{align*}
which implies 
\begin{align*}
g'(t)\begin{cases}
&<0\text{ \,for \,}0<t<t_1,\\
&=0\text{ \,for \,}t=t_1,\\
&>0\text{ \,for \,}t>t_1. 
\end{cases}
\end{align*}
Then recalling $\lim_{t\downarrow 0}g(t)=\infty$ and $\lim_{t\to\infty}g(t)=1$, 
we have $\alpha_v=\frac{1}{E_{1^*}}\inf_{t>0}g(t)=\frac{1}{E_{1^*}}g(t_1)>0$, 
which gives $0<\alpha_v<\frac{1}{E_{1^*}}$. 
On the other hand, we see 
\begin{align*}
\lim_{t\downarrow 0}h(t)=1\text{ \,and \,}
\lim_{t\to\infty}h(t)=\lim_{t\to\infty}\frac{(1+t)^{-\frac{1}{b}}}{1-(\frac{t}{1+t})^{\frac{1^*}{a}}}
=\frac{a}{b1^*}\lim_{t\to\infty}(1+t)^{-\frac{1}{b}+1}=\infty
\end{align*}
since $b>1$. Hence, we obtain $\alpha_c=\frac{1}{E_{1^*}}\sup_{t>0}h(t)=\infty$. 
As a result, $D_\alpha$ is attained for $\alpha>\alpha_v$, 
while $D_\alpha$ is not attained for $0<\alpha<\alpha_v$. 
Next, we consider the case $\alpha=\alpha_v$. 
Note $\alpha_v E_{1^*}=g(t_1)$ implies $f_{\alpha_v}(t_1)=1$. 
Combining this fact with $\lim_{t\downarrow 0}f_{\alpha_v}(t)=1$ 
and $\lim_{t\to\infty}f_{\alpha_v}(t)=\alpha_v E_{1^*}=g(t_1)<1$, 
we can conclude that $\sup_{t>0}f_{\alpha_v}(t)$ is attained, which is equivalent 
to the attainability of $D_{\alpha_v}$ by Theorem \ref{thm5}. 
Next, we prove the asymptotic behaviors of $\alpha_v$ on $a$ and $b$. 
First, we show $\lim_{b\to\infty}\alpha_v=0$. 
By a direct computation, we have for $b>1$, 
\begin{align*}
g(t_0)
=\left(
\frac{1}{1-(\frac{a}{1^*})^{\frac{b}{b-1}}}
\right)^{\frac{1^*}{a}}\left(1-\left(\frac{a}{1^*}\right)^{\frac{1}{b-1}}\right)\to 0
\end{align*}
as $b\to\infty$. Since $t_1$ is the minimum point of $g(t)$ for $t>0$, 
we have $0<g(t_1)<g(t_0)\to 0$ as $b\to\infty$, 
and thus it holds $\lim_{b\to\infty}g(t_1)=0$, 
which shows $\alpha_v=\frac{1}{E_{1^*}}g(t_1)\to 0$ as $b\to\infty$. 
Next, we show $\lim_{b\downarrow 1}\alpha_v=\frac{1}{E_{1^*}}$. 
First, since $g(t_1)<1$ for $b>1$, we obtain $\overline\lim_{b\downarrow 1}g(t_1)\leq1$. 
On the other hand, recall that $t_1$ satisfies $\tilde g(t_1)=0$, which implies 
$1+t_1=\left(1+\frac{at_1}{b1^*}\right)^b$. Plugging this relation to $g(t_1)$, we see for $b>1$, 
\begin{align}
&\notag g(t_1)=\left(\frac{a}{b1^*}+\frac{1}{t_1}\right)^{\frac{1^*}{a}}
\left(\frac{a}{b1^*}\right)^{\frac{(b-1)1^*}{a}}\left(\frac{1}{1+\frac{b1^*}{at_1}}\right)
\left(t_1+\frac{b1^*}{a}\right)^{\frac{(b-1)1^*}{a}}\\
&\label{bto1-beha2}\geq\left(\left(\frac{a}{1^*}\right)^{\frac{1^*}{a}}+o(1)\right)
\left(t_0+\frac{b1^*}{a}\right)^{\frac{(b-1)1^*}{a}} 
\end{align}
as $b\downarrow 1$, where we used $t_1>t_0\to\infty$ as $b\downarrow 1$, 
which gives $\lim_{b\downarrow 1}t_1=\infty$. Furthermore, we observe for $b>1$, 
\begin{align*}
&\left(t_0+\frac{b1^*}{a}\right)^{\frac{(b-1)1^*}{a}}
=\left(\frac{1^*}{a}\right)^{\frac{b1^*}{a}}
\left(1+\frac{\frac{b1^*}{a}-1}{(\frac{1^*}{a})^{\frac{b}{b-1}}}\right)^{\frac{(b-1)1^*}{a}}
=\left(\frac{1^*}{a}\right)^{\frac{b1^*}{a}}\left(1+o(1)\right)^{\frac{(b-1)1^*}{a}}
\end{align*}
as $b\downarrow 1$, and hence, it holds 
\begin{align}\label{bto1-beha1}
\lim_{b\downarrow 1}\left(t_0+\frac{b1^*}{a}\right)^{\frac{(b-1)1^*}{a}}=\left(\frac{1^*}{a}\right)^{\frac{1^*}{a}}
\end{align}
Combining \eqref{bto1-beha2} with \eqref{bto1-beha1}, we obtain 
$\underline\lim_{b\downarrow 1}g(t_1)\geq\left(\frac{a}{1^*}\right)^{\frac{1^*}{a}}\left(\frac{1^*}{a}\right)^{\frac{1^*}{a}}=1$. 
As a conclusion, we have $\lim_{b\downarrow 1}g(t_1)=1$, 
which yields $\lim_{b\downarrow 1}\alpha_v=\frac{1}{E_{1^*}}\lim_{b\downarrow 1}g(t_1)=\frac{1}{E_{1^*}}$. 
Next, we show $\lim_{a\downarrow 0}\alpha_v=\frac{1}{E_{1^*}}$. 
Since $g(t_1)<1$, we have $\overline\lim_{a\downarrow 0}g(t_1)\leq 1$. 
On the other hand, noting $t_1>t_0\to\infty$ as $a\downarrow 0$, we see 
\begin{align*}
g(t_1)=\left(1+\frac{1}{t_1}\right)^{\frac{1^*}{a}}\left(1-(1+t_1)^{-\frac{1}{b}}\right)\geq
1-(1+t_1)^{-\frac{1}{b}}\to 1
\end{align*}
as $a\downarrow 0$, and thus it holds $\underline\lim_{a\downarrow 0}g(t_1)\geq1$. 
As a conclusion, we obtain $\lim_{a\downarrow 0}g(t_1)=1$, which implies $\lim_{a\downarrow 0}\alpha_v=\frac{1}{E_{1^*}}$. 
Next, we show $\lim_{a\uparrow 1^*}\alpha_v=\frac{1}{bE_{1^*}}$. 
We write $t_1=t_1(a)$ for $a<1^*$. 
First, we claim $\lim_{a\uparrow 1^*}t_1(a)=0$. 
To this end, assume that $\overline\lim_{a\uparrow 1^*}t_1(a)=\overline t_1\in(0,\infty]$. 
Then we can pick up a sequence $\{a_j\}_{j=1}^\infty\in(0,1^*)$ such that $a_j\uparrow 1^*$ as $j\to\infty$ 
and $\lim_{j\to\infty}t_1(a_j)=\overline t_1$. Recall that $t_1(a_j)$ satisfies $\tilde g(t_1(a_j))=0$, which implies 
\begin{align}\label{a_j-rela}
\frac{1}{b}+\frac{1^*}{a_jt_1(a_j)}-\frac{1^*(1+t_1(a_j))^{\frac{1}{b}}}{a_jt_1(a_j)}=0
\end{align}
for each $j$. First, assume that $\overline t_1=\infty$. Then letting $j\to\infty$ in \eqref{a_j-rela}, we obtain $\frac{1}{b}=0$, which is a contradiction. 
Hence, it holds $\overline t_1\in(0,\infty)$. Now letting $j\to\infty$ in \eqref{a_j-rela} again, we have 
\begin{align*}
\frac{1}{b}+\frac{1}{\overline t_1}-\frac{(1+\overline t_1)^{\frac{1}{b}}}{\overline t_1}=0,
\end{align*} 
which yields $1+\overline t_1-(1+\frac{\overline t_1}{b})^b=0$. However, this is a contradiction since we can check $1+t-(1+\frac{t}{b})^b<0$ for $t>0$. 
As a result, we have $\overline\lim_{a\uparrow 1^*}t_1=0$, which is equivalent to $\lim_{a\uparrow 1^*}t_1=0$. 
Now we compute $g(t_1)$. By a direct computation, we see 
\begin{align*}
&g(t_1)=\frac{(1+t_1)^{\frac{1^*}{a}-\frac{1}{b}}}{t_1^{\frac{1^*}{a}-1}}\frac{(1+t_1)^{\frac{1}{b}}-1}{t_1}
=\frac{\frac{1}{b}+o(1)}{t_1^{\frac{1^*}{a}-1}}
\end{align*}
as $a\uparrow 1^*$, where we used $\lim_{a\uparrow 1^*}t_1=0$. Hence, in order to prove $\lim_{a\uparrow 1^*}g(t_1)=\frac{1}{b}$, 
it is enough to show $\lim_{a\uparrow 1^*}t_1^{\frac{1^*}{a}-1}=1$. 
Since $0<t_0<t_1\to 0$ as $a\uparrow 1^*$, we see for $a<1^*$ close enough to $1^*$, 
\begin{align*}
\left|\left(\frac{1^*}{a}-1\right)\log t_1\right|=\left(\frac{1^*}{a}-1\right)\log\frac{1}{t_1}
<\left(\frac{1^*}{a}-1\right)\log\frac{1}{t_0}\to 0
\end{align*}
as $a\uparrow 1^*$, which gives $\lim_{a\uparrow 1^*}\left(\frac{1^*}{a}-1\right)\log t_1=0$, and hence, it holds 
$\lim_{a\uparrow 1^*}t_1^{\frac{1^*}{a}-1}=1$. As a conclusion, we obtain $\lim_{a\uparrow 1^*}g(t_1)=\frac{1}{b}$, 
which shows $\lim_{a\uparrow 1^*}\alpha_v=\frac{1}{E_{1^*}}\lim_{a\uparrow 1^*}g(t_1)=\frac{1}{bE_{1^*}}$. 

\medskip

\noindent
(ii) Let $a<1^*$ and $b\leq1$. 
In the same way as in the case (i), we have for $t>0$, 
\begin{align*}
g'(t)=t^{-\frac{1^*}{a}-1}(1+t)^{\frac{1^*}{a}-\frac{1}{b}-1}\tilde g(t)
\text{ \,and \,}\tilde g'(t)=\frac{1}{b}-\frac{1^*}{ab}(1+t)^{\frac{1}{b}-1}, 
\end{align*}
where $\tilde g(t):=\frac{t}{b}+\frac{1^*}{a}\left(1-(1+t)^{\frac{1}{b}}\right)$. 
Then we see $b\tilde g'(t)=1-\frac{1^*}{a}(1+t)^{\frac{1}{b}-1}\leq 1-\frac{1^*}{a}<0$ 
for $t>0$ since $a<1^*$ and $b\leq 1$, and thus it holds $\tilde g'(t)<0$ for $t>0$. 
Then since $\lim_{t\downarrow 0}\tilde g(t)=0$, we obtain $\tilde g(t)<0$ for $t>0$, 
which implies $g'(t)<0$ for $t>0$. 
This fact together with $\lim_{t\to\infty}g(t)=1$, we have 
$\alpha_v=\frac{1}{E_{1^*}}\inf_{t>0}g(t)=\frac{1}{E_{1^*}}$. 
On the other hand, we see $\lim_{t\downarrow 0}h(t)=1$ and 
\begin{align*}
\lim_{t\to\infty}h(t)=\lim_{t\to\infty}\frac{
(1+t)^{-\frac{1}{b}}
}{
1-(\frac{t}{1+t})^{\frac{1^*}{a}}
}=\frac{a}{b1^*}\lim_{t\to\infty}(1+t)^{-\frac{1}{b}+1}=
\begin{cases}
&\frac{a}{1^*}\text{ \,when \,}b=1,\\
&0\text{ \,when \,}b<1. 
\end{cases}
\end{align*}
In the same way as in the case (i), we see for $t>0$, 
\begin{align*}
h'(t)=\frac{
t^{\frac{1^*}{a}}(1+t)^{\frac{1^*}{a}-\frac{1}{b}-1}
}{
b\left((1+t)^{\frac{1^*}{a}}-t^{\frac{1^*}{a}}\right)^2
}\tilde h(t)\text{ \,and \,}
\tilde h'(t)=\frac{1^*}{at^2}\left(\left(\frac{1+t}{t}\right)^{\frac{1^*}{a}-1}-b\right)>\frac{1^*}{at^2}(1-b)\geq0
\end{align*}
since $a<1^*$ and $b\leq1$, where $\tilde h(t):=\frac{b1^*}{at}+1-(\frac{1+t}{t})^{\frac{1^*}{a}}$. 
Hence, it follows $\tilde h'(t)>0$ for $t>0$. Since $\lim_{t\to\infty}\tilde h(t)=0$, 
we obtain $\tilde h(t)<0$ for $t>0$, which shows $h'(t)<0$ for $t>0$. 
This fact together with $\lim_{t\downarrow 0}h(t)=1$ and 
$\lim_{t\to\infty}h(t)=\begin{cases}
&\frac{a}{1^*}<1\text{ \,when \,}b=1,\\
&0\text{ \,when \,}b<1
\end{cases}
$ gives $\alpha_c=\frac{1}{E_{1^*}}\sup_{t>0}h(t)=\frac{1}{E_{1^*}}$. 
As a result, we have $\alpha_v=\alpha_c=\frac{1}{E_{1^*}}$, 
and then $D_\alpha$ is not attained for $\alpha\ne\alpha_v(=\alpha_c)$. 
Next, we consider the case $\alpha=\alpha_v$. 
Note that $(1=)\alpha_v E_{1^*}<g(t)$ for $t>0$ implies 
$f_{\alpha_v}(t)<1=\lim_{s\downarrow 0}f_{\alpha_v}(s)\leq\sup_{s>0}f_{\alpha_v}(s)$ 
for $t>0$. Hence, $\sup_{t>0}f_{\alpha_v}(t)$ is not attained, 
which is equivalent to the non-attainability of $D_{\alpha_v}$ by Theorem \ref{thm5}. 
The proof of Lemma \ref{last-lem-cri} is complete. 
\end{proof}

\begin{proof}[{\rm \bf Proof of Theorems \ref{thm3}-\ref{thm4}}]
Gathering up Lemmas \ref{a>1^*-ests}-\ref{last-lem-cri}  we have the results stated in Theorems \ref{thm3}-\ref{thm4}. 
\end{proof}

\end{document}